\documentclass[11pt,a4paper,english]{article}

\usepackage[latin1]{inputenc}
\usepackage{babel}
\usepackage[centertags]{amsmath}
\usepackage{amsfonts}
\usepackage{amssymb}
\usepackage{color}
\usepackage{amsthm}
\usepackage{newlfont}
\usepackage{graphicx}
\usepackage{dsfont}
\usepackage{geometry}
\usepackage{mathrsfs}
\usepackage{authblk}
\usepackage{verbatim}
\usepackage{hyperref}

\numberwithin{equation}{section}

\makeatletter
\renewcommand{\section}{\@startsection{section}{1}{0pt}{20pt}{6pt}{\large\bf}}
\renewcommand{\@seccntformat}[1]{\csname the#1\endcsname.\ }

\def\footnoterule{\kern -3pt \hrule width 2.7 true cm \kern 2.6pt}

\def\ni{\noindent}
\def\vs{\vspace}
\def\hs{\hspace}

\def\EE{\mathsf E}
\def\PP{\mathsf P}
\def\cF{{\cal F}}
\def\cD{{\cal D}}
\def\R{I\!\!R}
\def\L{I\!\!L}

\newcommand{\eps}{\varepsilon}
\newcommand{\p}{\! +\! }
\newcommand{\m}{\! -\! }

\newtheorem{theorem}{Theorem}[section]
\newtheorem{lemma}[theorem]{Lemma}
\newtheorem{coroll}[theorem]{Corollary}
\newtheorem{prop}[theorem]{Proposition}
\newtheorem{remark}[theorem]{Remark}
\newtheorem{ass}[theorem]{Assumption}

\newcommand{\msc}[1]{\textbf{MSC2010 Classification:} #1.}

\newcommand{\keywords}[1]{\textbf{Key words:} #1.}
\newcommand{\ackn}[1]{\textbf{Acknowledgments:} #1.}

\begin{document}

\title{\textbf{On the optimal exercise boundaries \\ of swing put options}}
\author{T. De Angelis\thanks{Corresponding author. School of Mathematics, University of Leeds, Woodhouse Lane,
Leeds LS2 9JT, UK; \texttt{t.deangelis@leeds.ac.uk}}\:\:\:\: \emph{and} \:\:\:\:Y. Kitapbayev\thanks{Questrom School of Business, Boston University, 595 Commonwealth Avenue, 02215, Boston, MA, USA; \texttt{yerkin@bu.edu}}
}
\date{\today}
\maketitle

%%%%%%%%%%%%%%%%%%%%%%%%%%%%%%%%%%%%%%%%%%%%%%%%%%%%%%%%%%%%%%%%%%%%%%%%%%%%%%%
%%% Abstract %%%
%%%%%%%%%%%%%%%%%%%%%%%%%%%%%%%%%%%%%%%%%%%%%%%%%%%%%%%%%%%%%%%%%%%%%%%%%%%%%%%

{\par \leftskip=2.6cm \rightskip=2.6cm \footnotesize
We use probabilistic methods to characterise time dependent optimal stopping boundaries in a problem of multiple optimal stopping  on a finite time horizon. Motivated by financial applications we consider a payoff of immediate stopping of ``put'' type and the underlying dynamics follows a geometric Brownian motion. The optimal stopping region relative to each optimal stopping time is described in terms of two boundaries which are continuous, monotonic functions of time and uniquely solve a system of coupled integral equations of Volterra-type. Finally we provide a formula for the value function of the problem.
\par}
\vspace{+8pt}

%%%%%%%%%%%%%%%%%%%%%%%%%%%%%%%%%%%%%%%%%%%%%%%%%%%%%%%%%%%%%%%%%%%%%%%%%%%%%%%%%%%%%%%%%%%%%
\msc{60G40, 60J60, 35R35, 91G20}
\vs{+5pt}

%\jel{C61, G13}
\vspace{+8pt}

\noindent\keywords{optimal multiple stopping, free-boundary problems, swing options, American put option}

%%%%%%%%%%%%%%%%%%%%%%%%%%%%%%%%%%%%%%%%%%%%%%%%%%%%%%%%%%%%%%%%%%%%%%%%%%%%%%%
%%% Sections %%%
%%%%%%%%%%%%%%%%%%%%%%%%%%%%%%%%%%%%%%%%%%%%%%%%%%%%%%%%%%%%%%%%%%%%%%%%%%%%%%%

%\vs{-18pt}

%\vs{-18pt}

\section{Introduction}
%%%%%%%%%%%%%%%%%%%%%%

In this paper we provide an analytical characterisation of the optimal stopping boundaries for a problem of optimal multiple stopping on finite time horizon. The study of this kind of problems has been recently motivated by the increasing popularity in financial industry of the so-called swing options. These are American-type options with multiple early exercise rights mostly used in the energy market (see \cite{Lem14} for a survey).

In particular here we consider a model for an option with put payoff, $n\in\mathbb{N}$, $n\ge2$ exercise rights, strike price $K>0$, maturity $T$ and \emph{refracting period} $\delta>0$. The parameter $\delta$ represents the minimum amount of time that the holder must wait between two consecutive exercises. The value of the option is denoted by $V^{(n)}$ and is given in terms of the following optimal multiple stopping time problem
\begin{align}\label{eq:V1}
V^{(n)}(t,x):=\sup_{\mathcal{S}^n_{t,T}}\EE \Big[\sum^n_{i=1}e^{-r(\tau_i-t)}\big(K\m X^{t,x}_{\tau_i}\big)^+\Big],\qquad(t,x)\in[0,T]\times (0,\infty)
\end{align}
where $r>0$ is the risk-free rate, $X^{t,x}$ is a geometric Brownian motion started at time $t\in[0,T)$ from $x>0$ and the optimisation is taken over the set of stopping times of $X^{t,x}$ of the form
\begin{align}\label{eq:StT}
\hs{-15pt}\mathcal{S}^n_{t,T}:=\big\{(\tau_n,\tau_{n-1},\ldots \tau_1):\tau_n\in[t,T\m(n\m 1)\delta],\tau_i\in[\tau_{i+1}\p\delta,T\m(i\m 1)\delta],i= n\m1,\ldots 1\big\}.%\nonumber
\end{align}
The index $k$ of the stopping time $\tau_k$ represents the number of remaining rights and the structure of $\mathcal{S}^n_{t,T}$ imposes to the option's holder to exercise all rights before the maturity $T$.
%Here we are also assuming that $e^{-r(s-t)}X^{t,x}_s$, $s\in[t,T]$ is a martingale.

In order to understand the financial meaning of problem \eqref{eq:V1} it is useful to observe that %since the discounted spot price is a martingale
\begin{align*}
\EE \Big[e^{-r(\tau_i-t)}\big(K\m X^{t,x}_{\tau_i}\big)^+\Big]=&\EE \Big[e^{-r(\tau_i-t)}\big(K\vee X^{t,x}_{\tau_i}\m X^{t,x}_{\tau_i}\big)\Big]
%=\EE \Big[ e^{-r(\tau_i-t)} \big(K\vee X^{t,x}_{\tau_i}\big)\Big]\m x
\end{align*}
for each $i=1,2,\ldots n$ and this payoff
%Hence solving \eqref{eq:V1} is equivalent to solving
%\begin{align}\label{max}
%\sup_{\mathcal{S}^n_{t,T}}\EE \Big[\sum^n_{i=1} e^{-r(\tau_i-t)} \big(K\vee X^{t,x}_{\tau_i}\big)\Big],
%\end{align}
can be used to model the following situation. In the energy market the seller of our option is a local energy supplier (for instance gas provider) and the buyer is a big extractor/distributor who trades on a global scale; both enjoy some storage facility. The local provider needs $n$ units of a commodity by time $T$ (for households' supply for instance) and agrees to buy these at the largest between the spot price $X$ and the strike $K$ on dates of the option holder's choosing. The holder commits to supplying the commodity by $T$ but can use the flexibility allowed by the contract to maximise profits. The value of this contract is therefore \eqref{eq:V1} because the option's holder sells a commodity with spot price $X$ and receives $X\vee K$. In this context the refracting time may also be due to physical constraints on the delivery. Additional details on the formulation of our problem are provided in Section \ref{sec:prob}.

We would like to emphasize that to date and to the best of our knowledge optimal boundaries of multiple stopping problems with finite time horizon have only been studied numerically, mostly in connections to swing options (cf.~for instance \cite{BLN11}, \cite{Car-Tou08}, \cite{Ib04} and \cite{LS09}) whereas problems on infinite horizon were studied theoretically by \cite{Car-Tou08} and \cite{Car-Da08}. These studies highlighted very intricate connections of recursive type: in particular the value and the optimal boundaries of a multiple stopping problem with $n$ admissible stopping times depend on those of all the problems with $k=1,2,\ldots, n-1$ admissible stopping times. As a consequence it turns out that each one of the latter problems must be solved prior to addressing the former one.

Since our work seems the first one addressing a fully theoretical characterisation of time-dependent optimal stopping boundaries for multiple stopping problems, the mathematical interest in the specific problem \eqref{eq:V1} finds natural motivations. Indeed very often the analysis of stopping boundaries for finite horizon problems with a single stopping time must be carried out on a case by case basis, due to the complexity of the methodologies involved. In this respect the American put is perhaps the most well-studied (and most popular) of such examples, and properties of its optimal boundary have been the object of a long list of papers (e.g. \cite{CJM}, \cite{Jacka} and \cite{Pe-0}). Problem \eqref{eq:V1} is therefore an ideal starting point for the analysis of free-boundary problems related to optimal multiple stopping. %from the conceptual point of view and will offer important insights for applications.

In this work we use probabilistic arguments to show that there exists a sequence $(\tau^*_i)_{i=n,\ldots 1}\in \mathcal{S}^n_{t,T}$ of optimal stopping times for \eqref{eq:V1} and we prove that each $\tau^*_i$, $i=n,\ldots 2$ is attained as the first exit time of the process $(t, X_t)$ from a set $C^{(i)}$. The latter is bounded from above and from below in the $(t,x)$-plane by two continuous monotonic curves, $b^{(i)}$ and $c^{(i)}$, functions of time (for $i=1$, $c^{(1)}=+\infty$ and $b^{(1)}$ is the American put optimal boundary). Our main results are the \emph{existence} and the \emph{regularity properties} of these optimal boundaries, which for the case of $n=2$ are given in Theorem \ref{thm:b2c2}, Proposition \ref{prop:finite-c} and Theorem \ref{thm:continuity} whereas their generalisation to any $n\ge2$ can be found in Section \ref{sec:gen}. Finally, for each $i=n,\ldots 2$, we characterise such boundaries as the unique solution of a system of coupled non-linear integral equations of Volterra type which we also solve numerically in some examples (see Figures $1$, $2$ and $3$). In line with the financial interpretation of problem \eqref{eq:V1} we show that the option's price is the sum of a European part and an early exercise premium which depends on the optimal stopping boundaries (see Theorem \ref{thm:eep2} for the case $n=2$ and Theorem \ref{thm:bdr-n} for the general case).

It is important to discuss the key difficulties of the free boundary analysis in \eqref{eq:V1} as these reflect more general theoretical questions that must be taken into account when studying problems of optimal multiple stopping.

It is known that \eqref{eq:V1} may be reduced via a recursive argument to a problem with a single stopping time (see Lemma \ref{lem:reduc} below) where the objective is to maximise a functional of the form $\EE e^{-r\tau}G^{(n)}(t+\tau, X^{t,x}_\tau)$ with suitable $G^{(n)}$. For each $n\ge 2$ the function $G^{(n)}$ depends on the value function $V^{(n-1)}$ of problem \eqref{eq:V1} with $n$ replaced by $n-1$ and it cannot be expressed explicitly as a function of $t$ and $x$ (see the discussion following Lemma \ref{lem:reduc} below). So for example in the case of $n=2$, the function $G^{(2)}$ will be a function of the American put value, denoted here $V^{(1)}$, and its most explicit form will be given in terms of a complicated functional of the American put optimal boundary $b^{(1)}$ (see the expressions \eqref{3.2} and \eqref{3.4} below). The latter is known to enjoy some monotonicity and regularity properties but their effect on $G^{(2)}$ is not easy to determine and no explicit formula for $b^{(1)}$ exists in the literature.

General probabilistic analysis of free boundaries associated to stopping problems in which the gain function is not given explicitly in terms of the state variables may be addressed in very few cases under ad-hoc assumptions. In \eqref{eq:V1} the gain function is dictated by the structure of the American put and we must compensate for the lack of transparency of $G^{(2)}$ with a thorough study of its regularity, and of an associated PDE problem (see Proposition \ref{rem:reguC^2}). Once that is accomplished we can use these results joint with fine estimates on the local time of the geometric Brownian motion to derive existence and other properties of the optimal boundaries.

Due to the unusual setting we work in, our preliminary study of $G^{(n)}$ and the proofs of our main results (in particular those of Theorem \ref{thm:b2c2} and Proposition \ref{prop:finite-c}) contain several technical points which extend the existing methods for the free-boundary analysis in optimal stopping theory and which we believe can be used to construct a more systematic approach to the study of optimal multiple stopping boundaries.

Finally, from a financial point of view the discovery of an upper optimal exercise boundary for our contract is an interesting result and it is in contrast with the single boundary observed in the American put problem. At a first sight this fact may look slightly counterintuitive but it turns out to be a consequence of the interplay between the time value of money and the constraints imposed on the set of stopping times $\mathcal{S}^n_{t,T}$ (see also Remark \ref{rem:optexn.1}). Here we also show that the size of the stopping set increases with the number of rights as conjectured in \cite{Car-Tou08} (see Remark \ref{rem:CarTou} below). This and other features will be discussed in fuller details in the rest of the paper.

The paper is organised as follows. In Section \ref{sec:swing} we provide a brief overview of the existing literature on optimal multiple stopping problems and their use in modeling swing contracts. Then in Section \ref{sec:prob} we introduce in details the setting of problem \eqref{eq:V1} outlined above. The full solution to our problem is given in Section \ref{sec:solution} which is split into two main subsections. Section \ref{sec:n=2} is devoted to the detailed analysis of a swing option with two exercise rights.
Instead we use Section \ref{sec:gen} to extend the results of Section \ref{sec:n=2} to the case of swing options with arbitrary many rights.
The paper is completed by a technical appendix.

\section{Formulation of the problem and background material}
We provide here some basic references on swing options and optimal multiple stopping and then we formulate problem \eqref{eq:V1} in details. In the last part of the section we recall some background material regarding the American put which we will use throughout the paper.

\subsection{An overview on swing options and optimal multiple stopping}\label{sec:swing}

Early mathematical models of swing contracts date back to the 80's (among others see \cite[Sec.~1 and 2]{JaiRonTom04} and references therein) and many authors have so far contributed to their development (see for instance the survey~\cite{Lem14}).
Numerical studies of swing options with volume constraints and limited number of trades at each exercise date may be found in \cite{JaiRonTom04} and \cite{Ib04}. Recently those works have been extended by \cite{BBP09}, \cite{Gobet}, \cite{HHK09}, \cite{WaLee11}, among others, to include more general dynamics of the underlying and complex structures of the options (for example jump dynamics and regime switching opportunities).

To the best of our knowledge a first theoretical analysis of the optimal stopping theory underpinning swing contracts was given in \cite{Car-Tou08} and it was based on martingale methods and Snell envelope. Later on a systematic study of martingale methods for multiple stopping time problems was provided in \cite{KQR-1} under the assumption of c\`adl\`ag positive processes. A characterisation of the related value functions in terms of excessive functions was given in \cite{Car-Da08} in the case of one-dimensional linear diffusions whereas duality methods were studied in \cite{MH04}, \cite{AH10} and \cite{Ben11}, among others.

In the Markovian setting variational methods and BSDEs techniques have been widely employed. In \cite{BLN11} for instance the HJB equation for a swing option with volume constraint is analysed both theoretically and numerically.
Variational inequalities for optimal multiple stopping problems have been studied for instance in \cite{LS09} in the (slightly different) context of evaluation of stock options and in \cite{LBM11} in an extension of results of \cite{Car-Tou08} to one-dimensional diffusions with jumps. A study of BSDEs with jumps related to swing options may be found instead in \cite{BPTW12}.

%%%%%%%%%%%%%%%%%%%%%%%%%%%%%%%%%%%%%%%%%%%%%%%%%%%%%%%%%

\subsection{Formulation of the problem}\label{sec:prob}

It will be convenient in the following to refer to the value function \eqref{eq:V1} as to the swing option \emph{price} or \emph{value}.

On a complete probability space $(\Omega, \cF, \PP)$ we consider the Black and Scholes model for the underlying asset dynamics
\begin{equation} \label{2.1}
 dX_{s}=rX_{s}\,ds+\sigma X_{s}\,dB_{s}\,,\;\;\; X_0=x>0
 \end{equation}
where $B$ is a standard Brownian motion, $r>0$ is the risk free-interest rate, and $\sigma>0$ is the volatility coefficient. We denote by $(\cF_s)_{s\ge 0}$ the natural filtration generated by $(B_s)_{s\ge0}$ completed with the $\PP$-null sets and by $(X^x_s)_{s\ge0}$ the unique strong solution of \eqref{2.1}. It is well known that for any $x>0$ it holds
\begin{align}\label{eq:GBM}%\hs{+6pc}
X^x_s=x\,e^{\sigma B_s+(r-\frac{1}{2}\sigma^2)s}\qquad \text{for $s\ge0$}
\end{align}
and the infinitesimal generator associated to $X$ is given by
\begin{align*}%\label{def:LX}%\hs{+4pc}
\L_X f(x):=rxf'(x)+\tfrac{1}{2}\sigma^2x^2f''(x)\qquad\text{for $f\in C^2(\R\,)$.}
\end{align*}

For the reader's convenience we recall here \eqref{eq:V1}:
\begin{align}\label{eq:V1bis}
V^{(n)}(t,x):=\sup_{\mathcal{S}^n_{t,T}}\EE \Big[\sum^n_{i=1}e^{-r(\tau_i-t)}\big(K\m X^{t,x}_{\tau_i}\big)^+\Big],\qquad(t,x)\in[0,T]\times (0,\infty)
\end{align}
where the supremum is taken over the set of stopping times of $X^{t,x}$ of the form
\begin{align*}%\label{eq:StTbis}
\hs{-15pt}\mathcal{S}^n_{t,T}:=\big\{(\tau_n,\tau_{n-1},\ldots \tau_1):\tau_n\in[t,T\m(n\m 1)\delta],\tau_i\in[\tau_{i+1}\p\delta,T\m(i\m 1)\delta],i= n\m1,\ldots 1\big\}.%\nonumber
\end{align*}
The function $V^{(n)}$ denotes the price of a swing option with a put payoff $(K-x)^+$, strike $K>0$, maturity $T>0$, $n$ exercise rights and refracting period $\delta>0$. Since $\delta>0$ is the option holder's minimum waiting time between two consecutive exercises of the option it is natural to consider $T$, $n$ and $\delta$ such that $T\ge (n\m 1)\delta$.

Notice that in \eqref{eq:V1bis} we denoted by $X^{t,x}$ the solution of \eqref{2.1} started at time $t>0$ with initial condition $X_t=x$. However in what follows we will often use that $X^{x}_s = X^{t,x}_{t+s}$ in law for any $s\ge0$. Moreover, since we are in a Markovian framework for any Borel-measurable real function $F$ we will often replace $\EE\big[ F(t\p s,X^x_s)\big]$ by $\EE_x\big[ F(t\p s,X_s)\big]$ and $\EE\big[ F(t\p s, X^x_{t+s})\big|\cF_t\big]$ by $\EE_x\big[ F(t\p s, X_{t+s})\big|\cF_t\big]=\EE_{X^x_t}\big[ F(t\p s,X_s)\big]$ where $\EE_x$ is the expectation under the measure $\PP_x(\,\cdot\,)=\PP(\,\cdot\,|X_0=x)$.

The peculiarity of \eqref{eq:V1bis} is embedded in the definition of the class of admissible stopping times $\mathcal{S}^{n}_{t,T}$ which sets the following constraint: the option's holder must exercise all rights. In other words if the $k$-th right is not used strictly prior to $T\m(k\m 1)\delta$ all subsequent rights can only be exercised at their maturity, i.e.~the holder remains with a portfolio of $k-1$ European put options with times to maturity $\delta,2\delta,\ldots (k\m1)\delta$. On the other hand, in case of an early exercise of the $k$-th right the holder gets an immediate payoff $(K-X)^+$ and remains with a swing option with $k\m 1$ exercise rights the earliest of which can be used after waiting the refracting period $\delta>0$.

Swing contracts including an obligation for the holder to use a minimum number of rights are traded in the energy market and have been analysed since the early papers \cite[Sec.~3]{Ib04} and \cite[Sec.~2.3.1]{JaiRonTom04}, amongst many others. Our formulation considers the limiting case in which all the rights must be exercised and can be motivated by the option's seller actual need for the underlying commodity as discussed in the introduction.

From a purely mathematical point of view this formulation is of interest as it is opposite to the one considered in \cite{Car-Tou08} where the holder has no obligation to use a minimum number of rights. The \emph{numerical} investigation of the option with finite maturity in \cite{Car-Tou08} shows that the optimal stopping region associated to each one of the admissible stopping times lies entirely below the continuation set and has a single exercise boundary below the strike $K$. Here instead we will see how the constraint on $\mathcal{S}^n_{t,T}$ may induce the option holder to use one of the rights even if the asset price $X$ is larger than the strike $K$ (i.e.~the put payoff equals zero) in order to maintain the future early exercise rights (see Remark \ref{rem:optexn.1} for further details).

Both our example and the one in \cite{Car-Tou08} are necessary intermediate steps towards the full solution in the general case of a swing contract with constraints on the minimum number of exercise dates.
Finally we notice that since the option in \cite{Car-Tou08} allows the holder more flexibility, its value provides an upper bound for $V^{(n)}$ in \eqref{eq:V1bis}.

\subsection{Background material on the American put}
Consistently with our definition of $V^{(n)}$ we note that for $n=1$ the value function $V^{(1)}$ coincides with the value function of the American put option with maturity $T>0$ and strike price $K>0$.
With a slight abuse of notation we also denote $V^{(0)}$ the price of the European put option with maturity $T>0$ and strike $K>0$. In our Markovian framework for $t\in[0,T]$ and $x>0$ we have
\begin{align}\label{2.5} %\hs{7pc}
V^{(0)}(t,x)=\EE\Big[e^{-r(T-t)}(K\m X^x_{T-t})^+\Big]
\end{align}
and
\begin{align}\label{2.6} %\hs{7pc}
&V^{(1)}(t,x)=\sup_{0\le\tau\le T-t}\EE\Big[e^{-r\tau}(K\m X^x_{\tau})^+\Big]
\end{align}
where $\tau$ is a $(\cF_t)$-stopping time.

We now recall some well known results about the American put problem (see e.g.~\cite[Sec.~25]{PS} and references therein) which will be used as building blocks of our approach. We define the sets
\begin{align} \label{2.7a} %\hs{4pc}
&C^{(1)}:=\ \{\, (t,x)\in[0,T)\! \times\! (0,\infty):V^{(1)}(t,x)>(K\m x)^+\, \} \\[3pt]
 \label{2.7b}&D^{(1)}:=\ \{\, (t,x)\in[0,T)\! \times\! (0,\infty):V^{(1)}(t,x)=(K\m x)^+\, \}
 \end{align}
and recall that the first entry time of $(t,X_t)$ into $D^{(1)}$ is an optimal stopping time in \eqref{2.5}. Moreover, there exists a unique continuous boundary $t\mapsto b^{(1)}(t)$ separating $C^{(1)}$ from $D^{(1)}$, with $0<b^{(1)}(t)<K$ for $t\in[0,T)$, and the stopping time
\begin{align*}%\label{2.8} %\hs{4pc}
\tau_{1}:=\inf\big\{0\leq s\leq T\m t\,:\,X^x_{s}\leq b^{(1)}(t\p s)\big\}
\end{align*}
is optimal in \eqref{2.5}. It is also well known that $V^{(1)}\in C^{1,2}$ in $C^{(1)}$ and it solves
\begin{align*}%\label{pdeV01}%\hs{+4pc}
\big(V^{(1)}_t+\L_XV^{(1)}-rV^{(1)}\big)(t,x)=0\qquad\text{for $x>b^{(1)}(t)$, $t\in[0,T)$.}
\end{align*}

The map $x\mapsto V^{(1)}_x(t,x)$ is continuous across the optimal boundary $b^{(1)}$ for all $t\in [0,T)$ (so-called \emph{smooth-fit} condition) and $\big|V_x\big|\le 1$ on $[0,T]\times(0,\infty)$ (cf.~\cite{PS} eq.~(25.2.15), p.~381 and notice that $V^{(1)}(t,\,\cdot\,)$ is decreasing). A change-of-variable formula (cf.~\cite{Pe-1}) then gives a representation of $V^{(1)}$ which we will frequently use in the rest of the paper, i.e.
\begin{align}\label{chvarV}
e^{-rs}V^{(1)}(t\p s,X^x_s)=V^{(1)}(t,x)-rK\int_0^s{e^{-ru}I(X^x_u\le b^{(1)}(t\p u))du}+M_{t+s}
\end{align}
for $s\in[0,T-t]$ and $x>0$, where $(M_{t+s})_{s\in[0,T-t]}$ is a continuous martingale (see \cite{PS} eq.~(25.2.63), p.~390).

The following remark will be needed in the proof of Proposition \ref{prop:finite-c} and we give it here as part of the background material.
\begin{remark}\label{rem:maturity}
Notice that in order to take into account for different maturities one should specify them in the definition of the value function, i.e.~for instance denoting $V^{(n)}(t,x;T)$, $n=0,1$, for the European/American put option with maturity $T$. However this notation is unnecessarily complex since what effectively matters in pricing put options is the time-to-maturity. In fact for fixed $x\in(0,\infty)$ and $\lambda>0$ the value at time $t\in[0,T]$ of a European/American put option with maturity $T$ is the same as the value of the option with maturity $T+\lambda$ but considered at time $t+\lambda$, i.e.~$V^{(n)}(t,x;T)=V^{(n)}(t+\lambda,x;T+\lambda)$, $n=0,1$.  In this work we mainly deal with a single maturity $T$ and simplify our notation by setting $V^{(n)}(t,x):=V^{(n)}(t,x;T)$.
\end{remark}

%%%%%%%%%%%%%%%%%%%%%%%%%%%%%%%%%%%%%%%%%%%%%%%%%%%%%%%%%%%%%%%%%%%%%
\section{Solution to the problem}\label{sec:solution}

Our first task is to rewrite problem \eqref{eq:V1bis} in a more canonical form according to the standard optimal stopping theory. For each $n\ge 2$, any $t\in\big[0,T-(n-1)\delta\big]$ and $x>0$, we define
\begin{align}\label{def:payoff-n01}
G^{(n)}(t,x):=(K-x)^++R^{(n)}(t,x)
\end{align}
where we have denoted
\begin{align}\label{def:payoff-n02}
R^{(n)}(t,x):=\EE\Big[e^{-r\delta}V^{(n-1)}(t\p \delta,X^x_\delta)\Big]
\end{align}
the expected discounted value of a swing option with $n-1$ exercise rights, available to the option holder after the refracting time $\delta$. The next result was proved in \cite[Thm.~2.1]{Car-Tou08} in a setting more general than ours and we refer the reader to that paper for its proof. One should notice that the constraint we imposed on $\mathcal{S}^n_{t,T}$ requires a trivial adjustment of the proof in \cite{Car-Tou08}.
\begin{lemma}\label{lem:reduc}
For each $n$, and any $(t,x)\in[0,T\m(n\m1)\delta]\times(0,\infty)$ the value function $V^{(n)}$ of \eqref{eq:V1bis} may be equivalently written as
\begin{align}\label{def:OS00}
V^{(n)}(t,x)=\hs{-2pc}\sup_{\hs{+2pc}0\le \tau\le T-(n-1)\delta}\EE\Big[e^{-r\tau}G^{(n)}(t\p \tau,X^x_\tau)\Big]
\end{align}
and $$\tau^*_n:=\inf\{0\le s \le T\m(n\m 1)\delta: V^{(n)}(t\p s,X_s)=G^{(n)}(t\p s,X_s)\}$$ is optimal in \eqref{def:OS00}.

Moreover, for fixed $n$ the sequence of optimal stopping times $(\tau^*_k)_{k=1,\ldots n}$ for problems \eqref{def:OS00} with value functions $V^{(k)}$, $k=1,\ldots n$ is optimal in the original formulation \eqref{eq:V1bis}.
\end{lemma}

The initial problem is now reduced to a problem with a single stopping time but the complexity of the multiple exercise structure has not disappeared and it has been encoded into the gain function $G^{(n)}$. Indeed it must be noted that $G^{(n)}$ depends in a non trivial recursive way, through the function $R^{(n)}$, on the value functions of the swing options with $n\m1,n\m2\ldots,1$ remaining rights.  The optimisation in \eqref{def:OS00} involves a single stopping time $\tau$ which in particular should be understood as $\tau_n$ from \eqref{eq:V1bis}.

Our aim is to characterise the sequence of optimal stopping times from Lemma \ref{lem:reduc} in terms of a sequence of optimal stopping sets whose boundaries are then analysed. For that we rely upon an iterative method: once the properties of $V^{(k)}$ and $\tau^*_{k}$ have been found, the function $G^{(k+1)}$ can be determined and we can address the study of $V^{(k+1)}$ and $\tau^*_{k+1}$. Unfortunately in our finite maturity setting there is no hope to determine explicitly how $G^{(n)}$ depends on $t$ and $x$. This makes problem \eqref{def:OS00} substantially more difficult than the standard American put option problem (in either finite or infinite horizon) and requires new methods of solution.

%%%%%%%%%%%%%%%%%%%%%%%%%%%%%%%%%%%%%%%%%%%%%%%%%%%%%%%%%%%%

\subsection{Analysis of the swing option with $n=2$}\label{sec:n=2}
%%%%%%%%%%%%%%%%%%%%%%%%%%%%%%%%%%%%%%%%%%%%%%%%%%%%%%%%%
In order to follow the idea given above of solving the problem by iteration we perform in this section a thorough analysis of problem \eqref{def:OS00} with $n=2$. Later we will generalise these results to any $n\ge 2$ by induction.

Here the main objectives are: $i)$ characterising the optimal stopping region in terms of two bounded continuous functions of time, i.e.~the optimal boundaries; $ii)$ providing an early-exercise premium (EEP) representation formula for the value function $V^{(2)}$;  $iii)$ proving that the couple of optimal boundaries is the unique solution of suitable equations.

We begin by studying fine regularity of the gain function and continuity of the value function in Section \ref{subs:VG}. Then in Section \ref{subs:CD} we prove existence and finiteness of two optimal stopping boundaries (cf.~Theorem \ref{thm:b2c2} and Proposition \ref{prop:finite-c}). We continue in Section \ref{subs:fb} by proving continuity of the boundaries and the smooth-fit property. Finally in Theorem \ref{thm:eep2} of Section \ref{subs:EEP} we provide the EEP representation of the option's value and integral equations for the optimal boundaries.

\subsubsection{Initial study of the gain function and of the value function}\label{subs:VG}
To simplify notation we set $T_\delta:=T\m\delta$, $G:=G^{(2)}$ and $R:=R^{(2)}$ (cf.~\eqref{def:payoff-n01} and \eqref{def:payoff-n02}), then for $t\in[0,T_\delta]$ and $x>0$ we have
\begin{align} \label{3.2}%\hs{+2pc}
G(t,x)=(K\m x)^+ +R(t,x)=(K\m x)^+ +e^{-r\delta} \EE V^{(1)}(t\p\delta,X^x_\delta)%\quad\text{for $(t,x)\in[0,T_{\delta}]\times(0,\infty)$}
 \end{align}
and
\begin{equation} \label{3.1}%\hs{+6pc}
V^{(2)}(t,x)=\sup \limits_{0\le\tau\le T_{\delta}-t}\EE e^{-r\tau}G(t\p\tau,X^x_\tau).%\qquad\text{for $(t,x)\in[0,T_{\delta}]\times(0,\infty)$.}
\end{equation}

In order to gain a better understanding of the properties of $G$ we first observe that $R$ may be rewritten as
\begin{align}\label{3.4}%\hs{+7pc}
R(t,x)=V^{(1)}(t,x)-rKg(t,x)\qquad\text{for $(t,x)\in[0,T_\delta]\times(0,\infty)$}
\end{align}
with
\begin{align}\label{3.5}%\hs{+7pc}
g(t,x):=\int_0^\delta{e^{-rs}\PP\big(X^x_{s}\le b^{(1)}(t\p s)\big)}ds
\end{align}
by taking expectations in \eqref{chvarV} with $s=\delta$. We also define $f:[0,T_\delta]\times(0,\infty)\to (0,\infty)$ by
\begin{align}\label{3.6}%\hs{+7pc}
f(t,x):=e^{-r\delta}\PP\big(X^x_{\delta}\le b^{(1)}(t\p\delta)\big).
\end{align}
In the next proposition we obtain important properties of $R$ which reflect the mollifying effect of the log-normal density function. The proof is collected in Appendix.
\begin{prop}\label{rem:reguC^2}
The function $R$ lies in $C^{1,2}((0,T_\delta)\times(0,\infty))$ and it solves
\begin{align}
&\hs{-1pc}\big(R_t+\L_XR-rR\big)(t,x)=-rKf(t,x) \quad \text{for $(t,x)\in(0,T_\delta)\times(0,\infty)$}\label{LR00}.
\end{align}
Moreover
\begin{align}
&\hs{-1pc}H(t,x):=(G_t\p\L_X G\m rG)(t,x)=-rK\big(I(x< K)+f(t,x)\big) \label{ltsf-3}
\end{align}
for $(t,x)\in(0,T_\delta)\times\big[(0,K)\cup(K,\infty)\big]$ and $t\mapsto H(t,x)$ is decreasing for all $x>0$ since $t\mapsto b^{(1)}(t)$ is increasing.
\end{prop}

An application of It\^o-Tanaka formula, \eqref{ltsf-3} and standard localisation arguments to remove the martingale term, give a useful representation of the expectation in \eqref{3.1}, i.e.
\begin{align} \label{3.8}
\EE e^{-r\tau} G(t\p \tau,X^x_\tau)=\;&G(t,x) +\EE\int_0^\tau e^{-ru}H(t\p u,X^x_u)du+\frac{1}{2}\EE\int_0^\tau e^{-ru}dL^K_u (X^x)
 \end{align}
for $(t,x)\in[0,T_\delta]\times(0,\infty)$ and any stopping time $\tau\in[0, T_\delta-t]$. Here $\big(L^K_u(X^x)\big)_{u\ge 0}$ is the local time process of $X^x$ at level $K$ and we have used that $H(t\p u,X^x_u)I(X^x_u\neq K)=H(t\p u,X^x_u)$ $\PP$-a.s.~for all $u\in[0,T_\delta-t]$.

\begin{remark}\label{rem:prop.1}
Proposition \eqref{rem:reguC^2} and the representation \eqref{3.8} are the starting point of our analysis of an optimal stopping rule. For $\delta>0$ the function $f$ is strictly negative in the whole state space. Hence $H(t,x)<0$ for all $(t,x)$ and the first integral in \eqref{3.8} may be seen as a running cost incurred by the option holder at all times for delaying the exercise of the option. The only incentive to wait comes from the integral with respect to the local time which increases whenever the process $X$ crosses the strike price $K$. So we can heuristically argue at this point that the option holder should exercise the option if the underlying price is ``too far'' from the strike price, and in particular even if the put part of the payoff is out-of-the-money.

We notice that for $\delta>0$ the process $(e^{-rt}R(t,X_t))_{t\ge0}$ is a strict supermartingale due to \eqref{LR00}. For $\delta=0$ instead one has $f(t,x)=I(x\le b^{(1)}(t))$ so that $(e^{-rt}R(t,X_t))_{t\ge0}$ behaves as a martingale for as long as $X$ stays above $b^{(1)}$. This observation in conjunction with \eqref{ltsf-3} and \eqref{3.8} implies that in absence of a refracting time the option holder does not incur a cost of waiting when the price is above the strike $K$, hence there is no incentive to exercise if the put part of the option is out-of-the-money. These considerations will be further expanded in Remark \ref{rem:optexn.1} below once a more rigorous analysis of the problem has been carried out.
\end{remark}

The continuation and stopping sets of problem \eqref{3.1} are given respectively by
\begin{align}  %\hs{4pc}
\label{3.9}&C^{(2)}:= \{\, (t,x)\in[0,T_{\delta})\! \times\! (0,\infty):V^{(2)}(t,x)>G(t,x)\, \} \\[3pt]
\label{3.10}&D^{(2)}:= \{\, (t,x)\in[0,T_{\delta}]\! \times\! (0,\infty):V^{(2)}(t,x)=G(t,x)\, \}.
\end{align}
Lemma \ref{lem:reduc} provides an optimal stopping time for \eqref{3.1} as
\begin{align} \label{3.11} %\hs{7pc}
\tau^*=\inf\ \{\ 0\leq s\leq T_{\delta}\m t:(t\p s,X^x_{s})\in D^{(2)}\ \}.
\end{align}
This can also be seen by standard arguments. In fact let $\tau:=\overline{\tau}\wedge(T_\delta-t)$ and $\overline{\tau}$ be arbitrary but fixed stopping time. Since the gain function $G$ is continuous on $[0,T_\delta]\times(0,\infty)$, dominated convergence theorem easily implies that $(t,x)\mapsto \EE e^{-r\tau}G(t+\tau,X^{x}_\tau)$ is continuous as well due to \eqref{eq:GBM}. Then $V^{(2)}$ must be at least lower semi-continuous as supremum of continuous functions and the standard theory of optimal stopping (cf.~for instance \cite[Corollary 2.9, Sec.~2]{PS}) confirms that \eqref{3.11} is the smallest optimal stopping time in \eqref{3.1}.

We can now begin our analysis of the value function $V^{(2)}$ by proving its continuity.
\begin{prop}\label{prop:contV}
The value function $V^{(2)}$ of \eqref{3.1} is continuous on $[0,T_{\delta}]\times(0,\infty)$. Moreover $x\mapsto V^{(2)}(t,x)$ is convex and  Lipschitz continuous with constant $L>0$ independent of $t\in[0,T_\delta]$.
\end{prop}
\begin{proof}
\emph{Step 1}. It follows from convexity of $x\mapsto V^{(1)}(t,x)$ and \eqref{3.2} that the map $x\mapsto G(t,x)$ is convex on $(0,\infty)$ for every $t\in[0,T_\delta]$ fixed. Now if we take any $t\in[0,T_\delta]$, $0<x<y$ and $\alpha\in(0,1)$ we have that
\begin{align*}
 \alpha V^{(2)}(t,x) + (1\m\alpha)V^{(2)}(t,y)&\ge   \sup_{0\le\tau\le T_\delta-t}\EE e^{-r\tau}\big[\alpha G(t\p\tau,X^{x}_\tau)+(1\m\alpha)G(t\p\tau,X^{y}_\tau)\big]\\[+3pt]
&\ge \sup_{0\le\tau\le T_\delta-t}\EE e^{-r\tau}G\left(t\p\tau,X^{\alpha x+(1-\alpha)y}_\tau\right)\\
&=V^{(2)}(t,\alpha x+(1\m\alpha)y)
\end{align*}
where we used the convexity of $G$ in $x$ and $\alpha X^x_\tau+(1-\alpha)X^y_\tau=X^{\alpha x+(1-\alpha)y}_\tau$. Hence the function $x\mapsto V^{(2)}(t,x)$ is convex on $(0,\infty)$ as well and therefore $x\mapsto V^{(2)}(t,x)$ is continuous on $(0,\infty)$ for every given and fixed $t\in [0,T_\delta]$.

Notice that $x\mapsto G(t,x)$ is also decreasing and Lipschitz continuous, uniformly with respect to $t\in[0,T_\delta]$. Indeed, since $-1 \le V^{(1)}_x\le0$ and $x\mapsto (K-x)^+$ is Lipschitz, we obtain for $t\in[0,T_\delta]$ and $0<x_1<x_2<\infty$
\begin{align}\label{lipG00}
0\le G(t,x_1)-G(t,x_2)&\le |x_2-x_1|+e^{-r\delta}\EE\big|X^{x_2}_\delta-X^{x_1}_\delta\big|\\[+3pt]
&=\big(x_2-x_1\big)\big(1+\EE e^{-r\delta}X^1_\delta\big)
=2\big(x_2-x_1\big).\nonumber
\end{align}
It then follows from \eqref{eq:GBM}, \eqref{lipG00} and the optional sampling theorem that
\begin{align*}%\label{lipV00}
0\le V^{(2)}(t,x_1)-V^{(2)}(t,x_2)&\le \sup_{0\le\tau\le T_\delta-t}\EE e^{-r\tau}\big[G(t\p\tau,X^{x_1}_\tau)-G(t\p\tau,X^{x_2}_\tau)\big]\\[+3pt]
&\le2(x_2-x_1)\sup_{0\le\tau\le T_\delta-t}\EE e^{-r\tau}X^{1}_\tau=2(x_2-x_1)\nonumber
\end{align*}
for $t\in[0,T_\delta]$ and $0<x_1<x_2<\infty$. Hence $x\mapsto V^{(2)}(t,x)$ is Lipschitz continuous with constant $L\in(0,2]$, uniformly with respect to time.
\vs{6pt}

\emph{Step 2}. It remains to prove that $t\mapsto V^{(2)}(t,x)$ is continuous on $[0,T_\delta]$ for $x\in(0,\infty)$. We first notice that for fixed $x>0$ the map $t\mapsto G(t,x)$ is decreasing since $t\mapsto V^{(1)}(t,x)$ is such and therefore $t\mapsto V^{(2)}(t,x)$ is decreasing as well by simple comparison. Take $0\le t_1<t_2\le T_\delta$ and $x\in(0,\infty)$, let $\tau_1=\tau^* (t_1,x)$ be optimal for $V^{(2)}(t_1,x)$ and set $\tau_2:=\tau_1 \wedge (T_\delta\m t_2)$. Then using \eqref{LR00}, the fact that $\tau_1\ge\tau_2$ $\PP$-a.s.~and the inequality $(K\m x)^+ -(K\m y)^+ \le (y\m x)^+$ for $x,y \in \R$, we find
\begin{align} \label{3.12} \hs{2pc}
0\le\; & V^{(2)}(t_1,x)-V^{(2)}(t_2,x)\\[+3pt]
\le\; &\EE e^{-r\tau_1} G(t_1\p\tau_1,X^x_{\tau_1})-\EE e^{-r\tau_2} G(t_2\p\tau_2,X^x_{\tau_2})\nonumber\\[+3pt]
\le\;& \EE e^{-r\tau_1}(X^{x}_{\tau_2}\m X^{x}_{\tau_1})^++\EE\Big[e^{-r\tau_1}R(t_1\p\tau_1,X^x_{\tau_1})-e^{-r\tau_1}R(t_2\p\tau_2,X^x_{\tau_2})\Big]\nonumber\\[+3pt]
\le\;&\EE e^{-r\tau_1}(X^{x}_{\tau_2}\m X^{x}_{\tau_1})^++R(t_1,x)-R(t_2,x)\nonumber\\
&-rK\EE\int^{\tau_2}_0e^{-rs}\big[f(t_1\p s,X^x_s)-f(t_2\p s,X^x_s)\big]ds.\nonumber
\end{align}
Taking now $t_2-t_1\to0$ one has that the first term of the last expression in \eqref{3.12} goes to zero by standard arguments (see e.g.~formulae~ (25.2.12)--(25.2.14), p.381 of \cite{PS}), the second one goes to zero by continuity of $V^{(1)}$ and $b^{(1)}$ and the third term goes to zero by dominated convergence and continuity of $f$.
\vs{+4pt}

Continuity of $V^{(2)}$ on $[0,T_\delta]\times(0,\infty)$ now follows by combining step 1 and step 2 above.
\end{proof}

%%%%%%%%%%%%%%%%%%%%%%%%%%%%%%%%%%%%%%%%%%%%%%%%%%%%%%%%%%%%%%%%
\subsubsection{Geometry of continuation and stopping sets}\label{subs:CD}

Notice that since $V^{(2)}$ and $G$ are continuous then $C^{(2)}$ is an open set and $D^{(2)}$ is a closed set (cf.~\eqref{3.9} and \eqref{3.10}). In the next proposition we obtain an initial insight on the structure of the set $D^{(2)}$ in terms of the set $D^{(1)}$ (cf.~\eqref{2.7b}).
\begin{prop}\label{prop:D2D1}
The restriction to $[0,T_\delta]$ of the stopping set $D^{(1)}$ of problem \eqref{2.6} is contained in the stopping set $D^{(2)}$  of problem \eqref{3.1}, i.e.
\begin{align}%\hs{+5pc}
D^{(1)}\cap\big([0,T_\delta]\times(0,\infty)\big)\subseteq D^{(2)}.
\end{align}
\end{prop}
\begin{proof}
Take any point $(t,x)\in [0,T_{\delta}]\times(0,\infty)$ and let $\tau=\tau^*(t,x)$ denote the optimal stopping time  for $V^{(2)}(t,x)$, then by using \eqref{3.2}, \eqref{LR00} and recalling that $f\ge 0$ we have
\begin{align*} %\label{3.13}%\hs{-2pc}
V^{(2)}(t,x)-V^{(1)}(t,x)&\le \EE e^{-r\tau}G(t\p\tau,X^x_\tau)-\EE e^{-r\tau}(K-X^x_{\tau})^+\\
&=\EE e^{-r\tau}R(t\p\tau,X^x_\tau)\nonumber\\
&=R(t,x)-rK\EE\int_0^\tau{e^{-rs}f(t\p s,X^x_s)ds}\le G(t,x)-(K\m x)^+.\nonumber
\end{align*}
It then follows that for any $(t,x)\in D^{(1)}$ with $t\in[0,T_\delta]$, i.e.~such that $V^{(1)}(t,x)=(K\m x)^+$, it must be $V^{(2)}(t,x)=G(t,x)$, hence $(t,x)\in D^{(2)}$.
\end{proof}

We now define the $t$-sections of the continuation and stopping sets of problem \eqref{3.1} by
\begin{align}%\hs{+5pc}
\label{def:t-secC}
C^{(2)}_t&:=\{\, x\in (0,\infty):V^{(2)}(t,x)>G(t,x)\, \}\\[+3pt]
\label{def:t-secD}
D^{(2)}_t&:=\{\, x\in (0,\infty):V^{(2)}(t,x)=G(t,x)\, \}
\end{align}
for $t\in[0,T_\delta]$ and prove the following
\begin{prop}\label{prop:incrCD}
For any $0\le t_1<t_2\le T_\delta$ one has $C^{(2)}_{t_2}\subseteq C^{(2)}_{t_1}$ (equivalently $D^{(2)}_{t_2}\supseteq D^{(2)}_{t_1}$), i.e.~the family $\{ C^{(2)}_t\,,\,t\in[0,T_\delta]\}$ is decreasing in $t$ (equivalently the family $\{ D^{(2)}_t\,,\,t\in[0,T_\delta]\}$ is increasing in $t$).
\end{prop}
\begin{proof}
Fix $0\le t_1<t_2 <T_\delta$ and $x\in (0,\infty)$, and set $\tau=\tau^*(t_2,x)$ optimal for $V^{(2)}(t_2,x)$. Then we have
\begin{align} \label{3.14}%\hs{-2pc}
V^{(2)}(&t_1,x)-V^{(2)}(t_2,x)\\
\ge\;& \EE e^{-r\tau}G(t_1\p\tau,X^x_\tau)-\EE e^{-r\tau}G(t_2\p\tau,X^x_\tau)=\EE e^{-r\tau}\big(R(t_1\p\tau,X^x_\tau)-R(t_2\p\tau,X^x_\tau)\big)\nonumber\\
=\;&R(t_1,x)-R(t_2,x)-rK\EE\int_0^\tau{e^{-rs}\big[f(t_1\p s,X^x_s)-f(t_2\p s,X^x_s)\big]ds}\nonumber\\
\ge\; &R(t_1,x)-R(t_2,x)=G(t_1,x)-G(t_2,x)\nonumber
\end{align}
where in the last inequality we used that $t\mapsto f(t,x)$ is increasing on $[0,T_\delta]$ by monotonicity of $b^{(1)}$ on $[0,T]$. It follows from \eqref{3.14} that $(t_2,x)\in C^{(2)}$ implies $(t_1,x)\in C^{(2)}$ and the proof is complete.
\end{proof}
\vs{6pt}

So far the analysis of the swing option has produced results which are somehow similar to those found in the standard American put option problem. In what follows instead we will establish that the structure of $C^{(2)}$ is radically different from the one of $C^{(1)}$ (cf.~\eqref{2.7a}) due to the coexistence of two optimal exercise boundaries.
In the rest of the paper we will require the next simple result, whose proof is omitted as it can be obtained by an application of It\^o-Tanaka formula, optional sampling theorem and observing that the process $X$ has independent increments.
\begin{lemma}\label{lem:loctime}
For any $\sigma\le\tau$ stopping times in $[0,T_\delta]$ one has
\begin{align}%\hs{-2pc}
\EE&\Big[\int_{\sigma}^{\tau}{e^{-rt}dL^K_t(X^x)}\Big|\cF_\sigma\Big]\\
&=\EE\Big[e^{-r\tau}\big|X^x_\tau-K\big|\Big|\cF_\sigma\Big]
-e^{-r\sigma}\big|X^x_\sigma-K\big|-rK\EE\Big[\int^{\tau}_{\sigma}e^{-rt} \textrm{sign}(X^x_t -K)dt \Big|\cF_\sigma\Big].\nonumber
\end{align}
\end{lemma}
Now we characterise the structure of the continuation region $C^{(2)}$.
\begin{theorem}\label{thm:b2c2}
There exist two functions $b^{(2)},\,c^{(2)}:[0,T_\delta]\to(0,\infty]$ such that $0< b^{(2)}(t)<K<c^{(2)}(t)\le\infty$ and $C^{(2)}_t=(b^{(2)}(t),c^{(2)}(t))$ for all $t\in[0,T_\delta]$. Moreover $b^{(2)}(t)\ge b^{(1)}(t)$ for all $t\in[0,T_\delta]$, $t\mapsto b^{(2)}(t)$ is increasing and $t\mapsto c^{(2)}(t)$ is decreasing on $[0,T_\delta]$ with
\begin{align}\label{eq:limbdry}%\hs{+9pc}
\lim_{t\uparrow T_\delta}b^{(2)}(t)=\lim_{t\uparrow T_\delta}c^{(2)}(t)=K.
\end{align}
\end{theorem}
\begin{proof}
The proof of existence is provided in 3 steps.
\vs{+6pt}

\emph{Step 1}. First we show that it is not optimal to stop at $x=K$. To accomplish that we use arguments inspired by \cite{Vill07}. Fix $\eps>0$, set
$\tau_\eps=\inf \{ u\ge 0: X^K_u \notin(K\m \eps,K\p\eps)\}$, take $t\in[0,T_\delta]$ and denote $s=T_\delta\m t$ then by \eqref{ltsf-3} and \eqref{3.8} we have that
\begin{align} \label{3.16} %\hs{2pc}
V^{(2)}(&t,K)-G(t,K)\\
\ge\;& \EE e^{-r\tau_{\eps}\wedge s}G(t\p\tau_{\eps}\wedge s,X^K_{\tau_{\eps}\wedge s})-G(t,K)\nonumber\\
=\;&\frac{1}{2}\EE\int_0^{{\tau_{\eps}\wedge s}} e^{-ru}dL^K_u (X^K)
-rK\EE\int_0^{{\tau_{\eps}\wedge s}} e^{-ru}\big(I(X^K_u \le K)+f(t\p u,X^K_u)\big)du\nonumber\\
\ge\;&\frac{1}{2}\EE\int_0^{{\tau_{\eps}\wedge s}} e^{-ru}dL^K_u (X^K)
-C_1 \EE(\tau_{\eps}\wedge s)\nonumber
\end{align}
for some constant $C_1 >0$. The integral involving the local time can be estimated by using It\^o-Tanaka's formula as follows
\begin{align} \label{3.17} %\hs{2pc}
\EE\int_0^{{\tau_{\eps}\wedge s}} &e^{-ru}dL^K_u (X^K)\\
=\;&\EE e^{-r(\tau_\eps\wedge s)} |X^K_{\tau_{\eps}\wedge s}-K|-
r K\,\EE\int_0^{{\tau_{\eps}\wedge s}} e^{-ru} \text{sign}(X^K_u -K) du\nonumber\\
\ge\;&\EE e^{-r(\tau_\eps\wedge s)} |X^K_{\tau_{\eps}\wedge s}-K|-C_2\,\EE(\tau_{\eps}\wedge s)\nonumber
\end{align}
with $C_2=rK$. Since $|X^K_{\tau_{\eps}\wedge s}-K|\le \eps$ it is not hard to see that for any $0<p<1$ we have
\begin{align*}
e^{-r(\tau_\eps\wedge s)}|X^K_{\tau_{\eps}\wedge s}-K|\ge e^{-rp(\tau_\eps\wedge s)}\frac{|X^K_{\tau_{\eps}\wedge s}-K|^p}{\eps^p}e^{-r(\tau_\eps\wedge s)}|X^K_{\tau_{\eps}\wedge s}-K|
\end{align*}
then by taking the expectation and using the integral version of \eqref{2.1} we get
\begin{align*} %\label{3.18a} %\hs{2pc}
\EE e^{-r(\tau_\eps\wedge s)} |X^K_{\tau_{\eps}\wedge s}-K|\ge\;&
\frac{1}{\eps^p}\EE\big|e^{-r\tau_{\eps}\wedge s} (X^K_{\tau_{\eps}\wedge s}-K)\big|^{1+p}\\
=\;&\frac{1}{\eps^p}\EE\Big|rK\int_0^{{\tau_{\eps}\wedge s}}e^{-ru}du+\sigma\int_0^{{\tau_{\eps}\wedge s}} e^{-ru}X^K_u dB_u\Big|^{1+p}.\nonumber
\end{align*}
We now use the standard inequality $|a+b|^{p+1}\ge \tfrac{1}{2^{p+1}}|a|^{p+1}-|b|^{p+1}$ for any $a,b\in \R$ (see.~e.g.~Ex.~5 in \cite[Ch.~8,~Sec.~50,~p.~83]{KF}) and Burkholder-Davis-Gundy (BDG) inequa\-li\-ty (see e.g.~\cite[p.~63]{PS}) to obtain
\begin{align}\label{3.18}%\hs{-3pc}
\EE e^{-r\tau_{\eps}\wedge s} |X^K_{\tau_{\eps}\wedge s}-K|\ge\;&\frac{1}{\eps^p 2^{p+1}}\EE\Big|\sigma\int_0^{{\tau_{\eps}\wedge s}} e^{-ru}X^K_u dB_u\Big|^{1+p}
-\frac{1}{\eps^p}\EE\Big|rK\int_0^{{\tau_{\eps}\wedge s}}e^{-ru}du\Big|^{1+p}\\
\ge\;&C_4\, \EE\Big|\sigma^2 \int_0^{{\tau_{\eps}\wedge s}} e^{-2ru}(X^K_u)^2 du\Big|^{(1+p)/2}
-C_3 \EE ({\tau_{\eps}\wedge s})^{1+p}\nonumber\\[+3pt]
\ge\;&C_4\, C_5\, \EE ({\tau_{\eps}\wedge s})^{(1+p)/2}
-C_3 \EE ({\tau_{\eps}\wedge s})^{1+p}\nonumber
\end{align}
for some constants $C_3 =C_3 (\eps,p)$, $C_4=C_4 (\eps,p)$, $C_5=C_5 (\eps.p)>0$. Since we are interested in the limit as $T_\delta-t\to0$ we take $s<1$, and combining \eqref{3.16}, \eqref{3.17} and \eqref{3.18} we get
\begin{align} \label{3.19} %\hs{1pc}
V^{(2)}(t,K)-&G(t,K)\ge C_4\, C_5\, \EE ({\tau_{\eps}\wedge s})^{(1+p)/2}
-(C_1 \p C_2\p C_3) \EE(\tau_{\eps}\wedge s)%\nonumber
\end{align}
for any $t\in[0,T_\delta)$ such that $s=T_\delta-t<1$.  Since $p\p 1<2$ it follows from \eqref{3.19} by letting $s\downarrow 0$ that there exists $t^*<T_\delta$ such that $V^{(2)}(t,K)>G(t,K)$ for all $t\in(t^*,T_\delta)$. Therefore $(t,K)\in C^{(2)}_t$ for all $t\in(t^*,T_\delta)$ and since $t\mapsto C^{(2)}_t$ is decreasing (cf.~Proposition \ref{prop:incrCD}) this implies $(t,K)\in C^{(2)}_t$ for all $t\in[0,T_\delta)$, i.e.~it is never optimal to stop when the underlying price $X$ equals the strike $K$.
\vs{6pt}

\emph{Step 2}. Now we study the portion of $D^{(2)}$ above the strike $K$ and show that it is not empty. For that we argue by contradiction and we assume that there are no points in the stopping region above $K$. Take $\eps>0$, $x\ge K+2\eps$ and $t\in[0,T_\delta)$ and we denote $\tau=\tau^*(t,x)$ the optimal stopping time for $V^{(2)}(t,x)$. As before we set $s=T_\delta-t$ to simplify notation and define $\sigma_\eps:=\inf\{u\ge0\,:\,X^x_{u}\le K\p\eps\}\wedge T_\delta$. Then by \eqref{ltsf-3} and \eqref{3.8} we get
\begin{align*} %\label{3.20} %\hs{-1pc}
V^{(2)}&(t,x)-G(t,x) \nonumber\\
=&\EE e^{-r\tau}G(t\p\tau,X^x_\tau)-G(t,x)\nonumber\\
\le&-rK\EE\int_0^\tau e^{-ru}f(t\p u,X^x_u)du+\frac{1}{2}\EE\int_0^\tau e^{-ru}dL^K_u (X^x)\nonumber\\
\le&-rK\EE\Big[ I(\tau<s)\int_0^\tau e^{-ru}f(t\p u,X^x_u)du\Big]-rK\EE\Big[ I(\tau=s)\int_0^s e^{-ru}f(t\p u,X^x_u)du\Big]\nonumber\\
&+\frac{1}{2}\EE\Big[I(\sigma_\eps<\tau)\int_{\sigma_\eps}^\tau e^{-ru}dL^K_u (X^x) \Big]\nonumber\\
=&-rK\EE\Big[\int_0^s e^{-ru}f(t\p u,X^x_u)du\Big]+rK\EE\Big[ I(\tau<s)\int_\tau^s e^{-ru}f(t\p u,X^x_u)du\Big]\nonumber\\
&+\frac{1}{2}\EE\Big[I(\sigma_\eps<\tau)\int_{\sigma_\eps}^\tau e^{-ru}dL^K_u (X^x) \Big]\nonumber%\\
\end{align*}
where we have used the fact that for $u\le\sigma_\eps$ the local time $L^K_u(X^x)$ is zero. Since we are assuming that it is never optimal to stop above $K$ then it must be $\big\{\tau<s\big\}\subset\big\{\sigma_\eps<s\big\}$. Obviously we also have $\big\{\sigma_\eps<\tau\big\}\subset\big\{\sigma_\eps<s\big\}$ and hence
\begin{align}\label{eq:DK00}
V^{(2)}&(t,x)-G(t,x)\\
\le&-rK\EE\Big[\int_0^s e^{-ru}f(t\p u,X^x_u)du\Big]\nonumber\\
&+\EE\Big[ I(\sigma_\eps<s)\Big(rK\int_\tau^s e^{-ru}f(t\p u,X^x_u)du+\frac{1}{2}\int_{\sigma_\eps}^s e^{-ru}dL^K_u (X^x)\Big) \Big]\nonumber\\
\le & -rK\EE\Big[\int_0^s e^{-ru}f(t\p u,X^x_u)du\Big]\nonumber\\
&+rK s \PP(\sigma_\eps<s)+\frac{1}{2}\EE\left[ I(\sigma_\eps<s)\EE\Big(\int_{\sigma_\eps}^{\sigma_\eps\vee s} e^{-ru}dL^K_u (X^x)\Big|\cF_{\sigma_\eps}\Big)\, \right]\nonumber%\\
\end{align}
where we have used $0\le f\le 1$ (cf.\ \eqref{3.6}) and the fact that $I(\sigma_\eps<s\big)$ is $\cF_{\sigma_\eps}$-measurable. From Lemma \ref{lem:loctime} with $\sigma=\sigma_\eps$ and $\tau=\sigma_\eps\vee s$ and by the martingale property of $(e^{-rt} X^x_t)_{t\ge 0}$ we get
\begin{align}\label{eq:DK00b}%\hs{-4pc}
\EE\Big[\int_{\sigma_\eps}^{\sigma_\eps\vee s}& e^{-ru}dL^K_u (X^x)\Big|\cF_{\sigma_\eps}\Big]\\
\le& 2K+\EE\big[e^{-r(\sigma_\eps\vee s)}X^x_{\sigma_\eps\vee s}\big|\cF_{\sigma_\eps}\big]-e^{-r\sigma_\eps}X^x_{\sigma_\eps}
+rK\EE\big[I(\sigma_\eps<s)\int^s_{\sigma_\eps}{e^{-rt}dt}\big]
\le3K.\nonumber
\end{align}
Combining \eqref{eq:DK00} and \eqref{eq:DK00b}
we finally obtain
\begin{align}\label{eq:DK01}\hs{-2pc}
V^{(2)}&(t,x)\m G(t,x)\le-rK\EE\Big[\int_0^s\hs{-4pt} e^{-ru}f(t\p u,X^x_u)du\Big]\p \left(\tfrac{3}{2}K\p rKs\right)\PP(\sigma_\eps<s).
\end{align}
To estimate $\PP\big(\sigma_\eps<s\big)$ it is convenient to set $\alpha:=\ln\left(\frac{x}{K+\eps}\right)$, $Y_t:=\sigma B_t+(r-\sigma^2/2)t$ and $Z_t:=-\sigma B_t+c\, t$ with $c:=r+\sigma^2/2$. Notice that $Y_t\ge -Z_t$ for $t\in[0,T_\delta]$ and hence
\begin{align}\label{eq:prob01}%\hs{-3pc}
\PP(\sigma_\eps<s)=&\PP\Big(\inf_{0\le u \le s}X^x_u\le K\p\eps\Big)=\PP\Big(\inf_{0\le u \le s}Y_u\le -\alpha\Big)\\[+3pt]
\le &\PP\Big(\inf_{0\le u \le s}-Z_u\le -\alpha\Big)=\PP\Big(\sup_{0\le u \le s}Z_u\ge \alpha\Big)
\le \PP\Big(\sup_{0\le u \le s}\big|Z_u\big|\ge \alpha\Big)\nonumber
\end{align}
where we also recall that $x\ge K\p 2\eps$ and hence $\alpha>0$. We now use Markov inequality, Doob's inequality and BDG inequality to estimate the last expression in \eqref{eq:prob01} and it follows that for any $p>1$
\begin{align}\label{eq:prob02}%\hs{-3pc}
\PP\Big(\sup_{0\le u \le s}\big|Z_u\big|\ge \alpha\Big)\le &\frac{1}{\alpha^p}\EE\sup_{0\le u\le s}\big|Z_u\big|^p
\le\frac{2^{p-1}}{\alpha^p} \Big(c s^p+\sigma^p\EE\sup_{0\le u\le s}\big|B_u\big|^p \Big)\le C_1\big(s^p+s^{p/2}\big)
\end{align}
with suitable $C_1=C_1(p,\eps,x)>0$. Collecting \eqref{eq:DK01} and \eqref{eq:prob02} we get
\begin{align}\label{eq:DK02}%\hs{-4pc}
V^{(2)}(t,x)\m G(t,x)\hs{-1pt}\le\hs{-1pt} s\Big(\hs{-1pt}C_2(s^{ p}\p s^{ p/2}) \p C_3\big(s^{p-1}+s^{p/2-1}\big)\m rK\EE\Big[\frac{1}{s}\hs{-3pt}\int_0^s \hs{-4pt}e^{-ru}f(t\p u,X^x_u)du\Big]\Big)
\end{align}
for suitable $C_2=C_2(p,\eps,x)>$ and $C_3=C_3(p,\eps,x)>0$.
We take $p>2$ and observe that in the limit as $s\downarrow 0$
we get
\begin{align}\label{eq:DK03}%\hs{-3pc}
-rK\EE&\Big[\frac{1}{s}\int_0^s e^{-ru}f(t\p u,X^x_u)du\Big]+C_2(s^{p}\p s^{p/2}) \p C_3\big(s^{p-1}+s^{p/2-1}\big)\to -rK f(T_\delta,x)
\end{align}
and therefore the negative term in \eqref{eq:DK02} dominates since $f(T_\delta,x)>0$ for all $x\in(0,\infty)$. From \eqref{eq:DK02} and \eqref{eq:DK03} we get a contradiction and by arbitrariness of $\eps$ we conclude that for any $x>K$ there must be $t<T_\delta$ large enough and such that $(t,x)\in D^{(2)}$.

We show now that $(t,x)\in D^{(2)}$ with $x>K$ implies $(t,y)\in D^{(2)}$ for any $y>x$. Take $y>x>K$ and assume $(t,y)\in C^{(2)}$. Set $\tau=\tau^*(t,y)$ optimal for $V^{(2)}(t,y)$ defined as in \eqref{3.11} and notice that the horizontal segment $[t,T_\delta]\times \{x\}$ belongs to $D^{(2)}$ by Proposition \ref{prop:incrCD}. Then the process $(t\p s,X^y_s)_{s\in[0,T_\delta-t]}$ cannot hit the horizontal segment $[t,T_\delta]\times \{K\}$ without entering into the stopping set. Hence by \eqref{ltsf-3} and \eqref{3.8} we have
\begin{align*}%\hs{-3pc}
V^{(2)}(t,y)=\EE e^{-r\tau}G(t\p\tau,X^{y}_\tau)= G(t,y)-rK\EE\Big[\int^\tau_0e^{-r s}f(t\p s,X^y_s)ds\Big]\le G(t,y),
\end{align*}
i.e.~it is optimal to stop at once at $(t,y)$, and therefore we get a contradiction. We then conclude that for each $t\in[0,T_\delta)$ there exists at most a unique point $c^{(2)}(t)>K$ such that $D^{(2)}_t\cap(K,\infty)=[c^{(2)}(t),\infty)$ with the convention that if $c^{(2)}(t)=+\infty$ the set is empty. We remark that for now we have only proven that $c^{(2)}(t)<+\infty$ for $t<T_\delta$ suitably large. Finiteness of $c^{(2)}$ will be provided in Proposition \ref{prop:finite-c} below.\vs{+6pt}

\emph{Step 3}. Now let us consider the set $\{(t,x)\in[0,T_\delta)\times(0,K]\}$. From Proposition \ref{prop:D2D1} it follows that for each $t\in[0,T_\delta)$ the set $D^{(2)}_t\cap(0,K)$ is not empty. Moreover by using arguments as in step 2 above one can prove that for any $x<K$ there exists $t<T_\delta$ such that $(t,x)\in D^{(2)}$, and that $x\in D^{(2)}_t\implies y\in D^{(2)}_t$ for $0<y\le x\le K$. The latter implies that for each $t\in[0,T_\delta)$ there exists a unique point $b^{(2)}(t)\in(0,K)$ such that $D^{(2)}_t\cap(0,K)=(0,b^{(2)}(t)]$.
\vs{6pt}

Steps 1, 2 and 3 above imply that $C^{(2)}_t=\big(b^{(2)}(t),c^{(2)}(t)\big)$ for all $t\in[0,T_\delta)$ and for suitable functions $b^{(2)},\,c^{(2)}:[0,T_\delta)\to (0,\infty]$. The fact that $b^{(2)}(t)\ge b^{(1)}(t)$ is an obvious consequence of Proposition \ref{prop:D2D1}. On the other hand Proposition \ref{prop:incrCD} implies that $t\mapsto b^{(2)}(t)$ is increasing whereas $t\mapsto c^{(2)}(t)$ is decreasing so that their left-limits always exist. Since $\lim_{t\uparrow T_\delta}c^{(2)}(t)\ge K$ and for any $x>K$ there exists $t<T_\delta$ with $(t,x)\in D^{(2)}$ (see step 2 above), then $\lim_{t\uparrow T_\delta}c^{(2)}(t)=K$. From a similar argument and step 3 above we also obtain $\lim_{t\uparrow T_\delta}b^{(2)}(t)=K$.
\end{proof}

\begin{remark}\label{rem:optexn.1}
\textbf{1.} The existence of an upper boundary is a key consequence of the constraints imposed by the structure of $\mathcal{S}^{n}_{t,T}$ in \eqref{eq:StT} (see also the discussion at the beginning of Section \ref{sec:prob}) and it nicely reflects the time value of the early exercise feature of the option. Indeed for $t<T_\delta$ the holder may find profitable to use the first right even if $X_t>K$ (i.e.~the put part of the immediate exercise payoff is zero) in order to maintain the opportunity of early exercising the remaining put option with maturity at $T$ (after the refracting period).

If for some $t<T_\delta$ the underlying price $X_t$ is too large, the holder does not believe that it will fall below $K$ prior to $T_\delta$. In this case delaying the exercise of the first right is likely to produce a null put payoff while at the same time reducing the value of the subsequent early exercise right. On the other hand, by using immediately the first right, the option holder will maximise at least the opportunities of an early exercise of the second option. It then becomes intuitively clear that while the holder of a standard American put option has nothing to lose in waiting as long as $X$ stays above $K$, for our swing contract things are different: waiting always costs to the holder in terms of the early exercises of future rights. Hence when the immediate put payoff of the first option is way too much ``out of the money'' it is better to get rid of it!
\vspace{+5pt}

\noindent \textbf{2.} We observe that it is $\PP$-almost surely optimal to exercise the first right of the swing option strictly before the maturity $T_\delta$ since $$\PP(X^x_t\in C^{(2)}_t\:\text{for all $t\in[0,T_\delta]$})\le \PP(X^x_{T_\delta}=K)=0.$$

\noindent \textbf{3.} It is known that as $r\to0$ the premium of early exercise for the American option vanishes thus meaning that $b^{(1)}\equiv0$ for $r=0$. Analogously, for $r=0$ there is no incentive in using the first right of the swing contract with $n=2$ at any time prior to $T_\delta$ so that $b^{(2)}\equiv 0$ and $c^{(2)}\equiv+\infty$. This fact will be clearly embodied in the pricing formula for $V^{(2)}$ in Theorem \ref{thm:eep2} below.
\end{remark}

In Theorem \ref{thm:b2c2} we have proven that $c^{(2)}(t)<\infty$ for all $t$ smaller than but ``sufficiently close'' to $T_\delta$. We now aim at strengthening this statement by proving that $c^{(2)}$ is indeed finite on $[0,T_\delta]$.
\begin{prop}\label{prop:finite-c}
For all $t\in[0,T_\delta]$ the upper boundary $c^{(2)}$ is finite, i.e.
\begin{align}\label{eq:cbdd}
\sup_{t\in[0,T_\delta]}c^{(2)}(t)<+\infty.
\end{align}
\end{prop}
\begin{proof}
The proof is provided in two steps.
\vs{+6pt}

\emph{Step 1}. Let us assume that \eqref{eq:cbdd} is violated and denote $t_0:=\sup\{t\in[0,T_\delta]\,:\,c^{(2)}(t)=+\infty\}$. Consider for now the case $t_0>0$ and note that since $t\mapsto c^{(2)}(t)$ is decreasing by Theorem \ref{thm:b2c2} then $c^{(2)}(t)=+\infty$ for all $t\in[0,t_0)$. The function $c^{(2)}$ is right-continuous on $(t_0,T_\delta]$, in fact for any $t\in(t_0,T_\delta]$ we take $t_n\downarrow t$ as $n\to\infty$ and the sequence $(t_n,c^{(2)}(t_n))\in D^{(2)}$ converges to $(t,c^{(2)}(t+))$, with $c^{(2)}(t+):=\lim_{s\downarrow t}c^{(2)}(s)$. Since $D^{(2)}$ is closed it must also be $(t,c^{(2)}(t))\in D^{(2)}$ and $c^{(2)}(t+)\ge c^{(2)}(t)$ by Theorem \ref{thm:b2c2}, hence $c^{(2)}(t+)=c^{(2)}(t)$ by monotonicity.

For $x\in(K,+\infty)$ we define the right-continuous inverse of $c^{(2)}$ by $t_c(x):=\sup\big\{t\in[0,T_\delta]\,:\,c^{(2)}(t)>x\big\}$ and observe that $t_c(x)\ge t_0$. Fix $\eps>0$ such that $\eps<\delta\wedge t_0$, then there exists $\overline{x}=\overline{x}(\eps)>K$ such that $t_c(x)-t_0\le \eps/2$ for all $x\ge \overline{x}$ and we denote $\theta=\theta(x):=\inf\big\{u\ge0\,:\, X^x_u\le \overline{x}\big\}\wedge T_\delta$. In particular we note that if $c^{(2)}(t_0+)=c^{(2)}(t_0)<+\infty$ we have $t_c(x)=t_0$ for all $x>c^{(2)}(t_0)$. We fix $t=t_0-\eps/2$, take $x>\overline{x}$ and set $\tau=\tau^*(t,x)$ the optimal stopping time for $V^{(2)}(t,x)$ (cf.~\eqref{3.11}). 

From \eqref{3.8} we get
\begin{align*}%\hs{-1pc}
V^{(2)}&(t,x)-G(t,x)\nonumber\\
=&\EE e^{-r\tau}G(t\p\tau,X^x_\tau)-G(t,x)\nonumber\\
\le& \EE\Big[- rK\int^\tau_0{\hs{-4pt}e^{-rs}f(t\p s,X^x_s)ds}\p \frac{1}{2}\int_0^\tau{\hs{-4pt}e^{-rs}dL^K_s(X^x)}\Big]\nonumber\\
\le&- rK\EE\Big[I(\tau\le \theta)\int^{\tau}_0{\hs{-4pt}e^{-rs}f(t\p s,X^x_s)ds}\Big]+\EE\Big[I(\tau>\theta)\frac{1}{2}\int_\theta^\tau{\hs{-4pt}e^{-rs}dL^K_s(X^x)}\Big]\nonumber
\end{align*}
where we have used that $L^K_s(X^x)=0$ for $s\le \theta$. Since $c^{(2)}(t)=+\infty$ for $t\in[t_0-\eps/2,t_0)$ and the boundary is decreasing then it must be $\{\tau\le\theta\}\subseteq\{\tau\ge\eps/2\}$. Hence we obtain
\begin{align}
\label{eq:bdd00} V^{(2)}&(t,x)-G(t,x)\\
\le&- rK\EE\Big[\int^{\eps/2}_0{\hs{-4pt}e^{-rs}f(t\p s,X^x_s)ds}\Big]\nonumber\\
&+\EE\Big[I(\tau>\theta)\Big(\frac{1}{2}\int_\theta^\tau{\hs{-4pt}e^{-rs}dL^K_s(X^x)}\p  rK\int^{\eps/2}_0{\hs{-4pt}e^{-rs}f(t\p s,X^x_s)ds}\Big)\Big]\nonumber\\
\le&- rK\EE\Big[\int^{\eps/2}_0{\hs{-4pt}e^{-rs}f(t\p s,X^x_s)ds}\Big]+\frac{1}{2}\EE\left[I(\tau>\theta)\EE\Big[\int_\theta^{\tau\vee \theta}{\hs{-4pt}e^{-rs}dL^K_s(X^x)}\Big|\cF_\theta\Big]\,\right]\nonumber\\
&+rK\frac{\eps}{2}\PP(\tau>\theta)\nonumber
\end{align}
where in the last inequality we have also used that $0\le f\le 1$ on $[0,T_\delta]\times(0,\infty)$.
We now estimate separately the two positive terms in the last expression of \eqref{eq:bdd00}. For the one involving the local time we argue as in \eqref{eq:DK00b}, i.e.~we use Lemma \ref{lem:loctime} and the martingale property of the discounted price to get
\begin{align*}%\label{eq:bdd01}
%I(&\tau>\theta)
\EE\Big(\int_\theta^{\tau\vee \theta}{\hs{-4pt}e^{-rs}dL^K_s(X^x)}\Big|\cF_\theta\Big)
\le 3K.
% I(\tau>\theta).
\end{align*}
Then for a suitable constant $C_1>0$ independent of $x$ we get
\begin{align}\label{eq:bdd02}
\EE\left[I(\tau>\theta)\EE\Big(\int_\theta^{\tau\vee \theta}{\hs{-4pt}e^{-rs}dL^K_s(X^x)}\Big|\cF_\theta\Big)\,\right]+rK\frac{\eps}{2}\PP(\tau>\theta)
\le C_1 \PP(\tau>\theta).
\end{align}
Observe now that on $\big\{\tau>\theta\big\}$ the process $X$ started at time $t=t_0-\eps/2$ from $x>\overline{x}$ must hit $\overline{x}$ prior to time $t_0+\eps/2$, hence, for $c=r+\sigma^2/2$, we obtain
\begin{align}\label{eq:bdd03}%\hs{+3pc}
\PP(\tau>\theta)\le \PP\big(\inf_{0\le t\le \eps} X^x_t<\overline{x}\big)\le\PP\left(\inf_{0\le t\le \eps} B_t<\frac{1}{\sigma}\Big(\ln\frac{\overline{x}}{x}+c\,\eps\Big)\right).
\end{align}
Introduce another Brownian motion by taking $W:=-B$, then from \eqref{eq:bdd03} and the \emph{reflection principle} we find
\begin{align}\label{eq:bdd04}
\PP(\tau>\theta)\le&\PP\left(\sup_{0\le t\le \eps} W_t>-\frac{1}{\sigma}\Big(\ln\frac{\overline{x}}{x}+c\,\eps\Big)\right)=
2\PP\left(W_\eps>-\frac{1}{\sigma}\Big(\ln\frac{\overline{x}}{x}+c\,\eps\Big)\right)\\[+3pt]
=&2\Big[1-\Phi\Big(\tfrac{1}{\sigma\sqrt{\eps}}\big(\ln(x/\,\overline{x})-c\,\eps\big)\Big)\Big]=
2\Phi\Big(\tfrac{1}{\sigma\sqrt{\eps}}\big(\ln(\overline{x}/x)+c\,\eps\big)\Big)\nonumber
\end{align}
with $\Phi(y):=1/\sqrt{2\pi}\int^y_{-\infty}e^{-z^2/2}dz$ for $y\in\R$ and where we have used $\Phi(y)=1-\Phi(-y)$, $y\in\R$.

Going back to \eqref{eq:bdd00} we aim at estimating the first term in the last expression. For that we use Markov property to obtain
\begin{align}\label{eq:bdd05}%\hs{-2pc}
\EE f(t\p s,X^x_s)=e^{-r\delta}\EE\big[\PP\big(X^x_{s+\delta}\le b^{(1)}(t\p s\p \delta)\big|\cF_s\big)\,\big]=e^{-r\delta}\PP\big[X^x_{s+\delta}\le b^{(1)}(t\p s\p \delta)\big]
\end{align}
with $s\in[0,\eps/2]$. For all $x>\overline{x}$ and $s\in[0,\eps/2]$ and denoting $\alpha:=b^{(1)}(t\p \delta)$, the expectation in \eqref{eq:bdd05} is bounded from below by recalling that $b^{(1)}$ is increasing, namely
\begin{align}\label{eq:bdd06}%\hs{-2pc}
\EE f(t\p s,X^x_s)\ge& e^{-r\delta}\PP(X^x_{s+\delta}\le \alpha)\ge e^{-r\delta}\PP\Big(B_{s+\delta}\le \frac{1}{\sigma}\big[\ln(\alpha/x)-c(\delta\p \eps/2)\big]\Big)\\[+3pt]
=&e^{-r\delta}\Phi\Big(\tfrac{1}{\sigma\sqrt{\delta\p s}}\big[\ln(\alpha/\,x)-c\,(\delta\p \eps/2)\big]\Big)\nonumber\\[+3pt]
\ge &e^{-r\delta}\Phi\Big(\tfrac{1}{\sigma\sqrt{\delta}}\big[\ln(\alpha/\,x)-c\,(\delta\p \eps/2)\big]\Big)=:\hat{F}(x)\nonumber
\end{align}
where in the last inequality we have used that $\ln(\alpha/\,x)<0$ and $\Phi$ is increasing. From \eqref{eq:bdd06}, using Fubini's theorem we get
\begin{align}\label{eq:bdd07bis}%\hs{-2pc}
\EE&\Big[\int^{\eps/2}_0{\hs{-4pt}e^{-rs}f(t\p s,X^x_s)ds}\Big]=\int^{\eps/2}_0{\hs{-4pt}e^{-rs}\EE f(t\p s,X^x_s)\,ds}\ge \frac{\eps}{2}e^{-r\eps/2}\hat{F}(x)
\end{align}
for $x>\overline{x}$. We now collect bounds \eqref{eq:bdd00}, \eqref{eq:bdd02}, \eqref{eq:bdd04} and \eqref{eq:bdd07bis} to obtain
\begin{align}\label{eq:bdd07}%\hs{-2pc}
V^{(2)}&(t,x)-G(t,x)\\
\le& 2C_1\Phi\Big(\tfrac{1}{\sigma\sqrt{\eps}}\big(\ln(\overline{x}/x)+c\,\eps\big)\Big)-C_2
\Phi\Big(\tfrac{1}{\sigma\sqrt{\delta}}\big[\ln\big(\alpha/\,x\big) -c\,(\delta\p \eps/2)\big]\Big)\nonumber
\end{align}
where $C_2=C_2(\eps)>0$ and independent of $x$. Since $t,\overline{x},\eps$ are fixed with $\delta>\eps$, we take the limit as $x\to\infty$ and it is not hard to verify by L'H\^opital's rule that
\begin{align*}\hs{-3pc}
\lim_{x\to\infty}&\frac{\Phi\Big(\tfrac{1}{\sigma\sqrt{\eps}}\big(\ln(\overline{x}/x)+c\,\eps\big)\Big)}
{\Phi\Big(\frac{1}{\sigma\sqrt{\delta}}\big[\ln\big(\alpha/\,x\big) -c\,(\delta\p \eps/2)\big]\Big)}=C_3\lim_{x\to\infty}
\frac{\varphi\Big(\frac{1}{\sigma\sqrt{\eps}}\big(\ln(\overline{x}/x)+c\,\eps\big)\Big)}
{\varphi\Big(\frac{1}{\sigma\sqrt{\delta}}\big[\ln\big(\alpha/\,x\big) -c\,(\delta\p \eps/2)\big]\Big)}\\[+3pt]
= & C_4\lim_{x\to\infty}x^\beta\exp\Big(\frac{1}{\sigma^2}\big(1/\delta-1/\eps\big)\big(\ln\,x\big)^2 \Big)= 0\nonumber
\end{align*}
for suitable positive constants $\beta>0$, $C_3$ and $C_4$ and with $\varphi:=\Phi'$ the standard normal density function. Hence the negative term in \eqref{eq:bdd07} dominates for large values of $x$ and we reach a contradiction. That implies $c^{(2)}(t)<+\infty$ for all $t\in (0,T_\delta]$ by arbitrariness of $t_0$.
\vspace{+6pt}

\emph{Step 2}. It remains to show that $c^{(2)}(0)<+\infty$ as well. In order to do so we recall Remark \ref{rem:maturity} and notice that $V^{(2)}(0,x;T_\delta)-G(0,x;T_\delta)=V^{(2)}(\lambda,x;T_\delta+\lambda)-G(\lambda,x;T_\delta+\lambda)$ for $\lambda>0$. Hence the arguments in step 1 may be applied with $t_0=\lambda$ and $T_\delta$ replaced by $T_\delta+\lambda$, proving that indeed $c^{(2)}(0)<+\infty$.
\end{proof}

%%%%%%%%%%%%%%%%%%%%%%%%%%%%%%%%%%%%%%%%%%%%%%%%%%
\subsubsection{Free-boundary problem for $V^{(2)}$ and continuity of the boundaries}\label{subs:fb}

To prepare the ground to the free-boundary problem for $V^{(2)}$ we begin by showing in the next proposition that the value function $V^{(2)}$ fulfills the so-called \emph{smooth-fit} condition at both the optimal boundaries $b^{(2)}$ and $c^{(2)}$.
\begin{prop}\label{prop:smooth-fit}
For all $t\in[0,T_\delta)$ the map $x\mapsto V^{(2)}(t,x)$ is $C^{1}$ across the optimal boundaries, i.e.
\begin{align}\label{3.22}%\hs{6pc}
&V^{(2)}_x (t, b^{(2)}(t)+)=G_x (t, b^{(2)}(t))\\[+3pt]
\label{3.22b}
&V^{(2)}_x (t, c^{(2)}(t)-)=G_x (t, c^{(2)}(t)).
\end{align}
\end{prop}
\begin{proof}
We provide a full proof only for \eqref{3.22b} as the one for \eqref{3.22} can be obtained in a similar way. Fix $0\le t<T_\delta$ and set $x_0 :=c^{(2)}(t)$. It is clear that for arbitrary $\varepsilon>0$ it hods
\begin{align*}%\label{3.23}%\hs{2pc}
\frac{V^{(2)}(t,x_0)-V^{(2)}(t,x_0\m\varepsilon)}{\varepsilon}\le \frac{G(t,x_0)-G(t,x_0\m\varepsilon)}{\varepsilon}
\end{align*}
and hence
\begin{align}\label{3.24}%\hs{2pc}
\limsup_{\varepsilon\to0}\frac{V^{(2)}(t,x_0)-V^{(2)}(t,x_0\m\eps)}{\varepsilon}\le G_x (t, x_0).
\end{align}
To prove the reverse inequality, we denote $\tau_\eps=\tau^*(t,x_0 -\eps)$ which is the optimal stopping time for $V^{(2)}(t,x_0 -\varepsilon)$. Then using the law of iterated logarithm at zero for Brownian motion and the fact that $t\mapsto c^{(2)}(t)$ is decreasing we obtain $\tau_\eps\to0$ as $\varepsilon\to0$, $\PP$-a.s. An application of the mean value theorem gives
\begin{align*}%\label{3.26}%\hs{2pc}
\frac{1}{\varepsilon}\Big(V^{(2)}&(t,x_0)-V^{(2)}(t,x_0\m\varepsilon)\Big)\\
\ge\;& \frac{1}{\eps}\EE\Big[e^{-r\tau^\eps}\Big(G(t\p\tau_\eps,X^{x_0}_{\tau_\eps})-
G(t\p\tau_\eps,X^{x_0-\eps}_{\tau_\eps})\Big)\Big]\nonumber\\
\ge\;&\frac{1}{\eps}\EE\Big[e^{-r\tau_\eps}G_x (t\p\tau_\eps,\xi)\big(X^{x_0}_{\tau_\eps}-X^{x_0 -\eps}_{\tau_\eps}\big)\Big]=\;\EE\Big[e^{-r\tau_\eps}G_x (t\p\tau_\eps,\xi)X^{1}_{\tau_\eps}\Big]\nonumber
\end{align*}
with $\xi(\omega)\in[X^{x_0 -\eps}_{\tau_{\varepsilon}}(\omega),X^{x_0}_{\tau_{\varepsilon}}(\omega)]$ for all $\omega\in\Omega$. Thus recalling that $G_x$ is bounded (cf.~\eqref{lipG00}) and $X^1_{\tau_\eps} \rightarrow 1$ $\PP$-a.s.~as $\eps\to0$, using dominated convergence theorem we obtain
\begin{align}\label{3.27}\hs{2pc}
\liminf_{\varepsilon\to0}\frac{V^{(2)}(t,x_0)-V^{(2)}(t,x_0 \m \eps)}{\varepsilon}\ge G_x (t, x_0).
\end{align}
Finally combining \eqref{3.24} and \eqref{3.27}, and using that $V(t,\,\cdot\,)$ is convex (see Proposition \ref{prop:contV}) we obtain \eqref{3.22b}.
\end{proof}
\vs{6pt}

Standard arguments based on the strong Markov property and continuity of $V^{(2)} $(see~\cite[Sec.~7]{PS}, p.~131) imply that $V^{(2)}\in C^{1,2}$ inside the continuation set $C^{(2)}$ and it solves the following free-boundary problem
\begin{align}
\label{3.31}&V^{(2)}_t \p\L_X V^{(2)}\m rV^{(2)}=0 &\hs{-30pt}\text{in $C^{(2)}$}\\
\label{3.32}&V^{(2)}(t,b^{(2)}(t))=G(t,b^{(2)}(t)) &\hs{-30pt}\text{for $t\in[0,T_\delta]$}\\
\label{3.33}&V^{(2)}(t,c^{(2)}(t))=G(t,c^{(2)}(t)) &\hs{-30pt}\text{for $t\in[0,T_\delta]$}\\
\label{3.34}&V^{(2)}_x (t,b^{(2)}(t)+)=G_x(t,b^{(2)}(t)) &\hs{-30pt}\text{for $t\in[0,T_\delta)$}\\
\label{3.35}&V^{(2)}_x (t,c^{(2)}(t)-)=G_x(t,c^{(2)}(t)) &\hs{-30pt}\text{for $t\in[0,T_\delta)$}\\
\label{3.36}&V^{(2)}(t,x)\ge G(t,x) &\hs{-30pt}\text{for}\;(t,x)\in [0,T_\delta]\times(0,\infty).%\\
\end{align}
Thanks to our results of the previous section the continuation set $C^{(2)}$ and the stopping set $D^{(2)}$ are given by
\begin{align} \label{3.38} %\hs{4pc}
&C^{(2)}= \{\, (t,x)\in[0,T_{\delta})\! \times\! (0,\infty):b^{(2)}(t)<x<c^{(2)}(t)\, \} \\[3pt]
\label{3.39}&D^{(2)}= \{\, (t,x)\in[0,T_{\delta}]\! \times\! (0,\infty):x\le b^{(2)}(t) \;
\text{or}\; x\ge c^{(2)}(t)\, \}.
\end{align}
Notice that \eqref{3.32}--\eqref{3.36} follow by the definition of $b^{(2)}$ and $c^{(2)}$ and by Proposition \ref{prop:smooth-fit}. For the unfamiliar reader we give in appendix the standard proof of \eqref{3.31}.

We now proceed to prove that the boundaries $b^{(2)}$ and $c^{(2)}$ are indeed continuous functions of time and follow the approach proposed in \cite{DeA14}.
\begin{theorem}\label{thm:continuity}
The optimal boundaries $b^{(2)}$ and $c^{(2)}$ are continuous on $[0,T_\delta]$.
\end{theorem}
\begin{proof}
The proof is provided in 3 steps.
\vs{+6pt}

\emph{Step 1}. By monotonicity and boundedness of $b^{(2)}$ and $c^{(2)}$ we obtain their right-continuity (see the first paragraph of step 1 in the proof of Proposition \ref{prop:finite-c}).
%Arguments similar to those used in the first paragraph of step 1 in the proof of Proposition \ref{prop:finite-c} allow us to easily show that $b^{(2)}$ and $c^{(2)}$ are right-continuous on $[0,T_\delta)$. Note that here the argument is even easier since we know that both boundaries are finite-valued.
\vs{+6pt}

\emph{Step 2}. Now we prove that $b^{(2)}$ is also left-continuous. Assume that there exists $t_0\in(0,T_\delta]$ such that $b^{(2)}(t_0-)<b^{(2)}(t_0)$ where $b^{(2)}(t_0-)$ denotes the left-limit of $b^{(2)}$ at $t_0$. Take $x_1<x_2$ such that $b^{(2)}(t_0-)<x_1<x_2<b^{(2)}(t_0)$ and $h>0$ such that $t_0>h$, then define the domain $\cD:=(t_0-h,t_0)\times (x_1,x_2)$ and denote by $\partial_P \cD$ its parabolic boundary formed by the horizontal segments $[t_0-h,t_0]\times\{x_i\}$ with $i=1,2$ and by the vertical one $\{t_0\}\times(x_1,x_2)$. Recall that both $G$ and $V^{(2)}$ belong to $C^{1,2}(\cD)$ and it follows from \eqref{ltsf-3}, \eqref{3.31} and \eqref{3.32} that $u:=V^{(2)}-G$ is such that $u\in C^{1,2}(\cD)\cap C(\overline{\cD})$ and it solves the boundary value problem
\begin{align}\label{cont00}\hs{+4pc}
u_t+\L_Xu-ru=-H\quad\text{on $\cD$ with $u=0$ on $\partial_P\cD$.}
\end{align}
Denote by $C^{\infty}_c(a,b)$ the set of continuous functions which are differentiable infinitely many times with continuous derivatives and compact support on $(a,b)$. Take $\varphi\in C^{\infty}_c(x_1,x_2)$ such that $\varphi\ge0$ and $\int^{x_2}_{x_1}{\varphi(x)dx}=1$. Multiplying \eqref{cont00} by $\varphi$ and integrating by parts we obtain
\begin{align}\label{cont01}
\int^{x_2}_{x_1}{\varphi(x)u_t(t,x)dx}=-\int^{x_2}_{x_1}{u(t,x)\left(\L_X^*\varphi(x)-r\varphi(x)\right)dx}-\int^{x_2}_{x_1}
{H(t,x)
\varphi(x)dx}
\end{align}
for $t\in(t_0\m h,t_0)$ and with $\L_X^*$ denoting the formal adjoint of $\L_X$. Since $u_t\le 0$ in $\cD$ by \eqref{3.14} in the proof of Proposition \ref{prop:incrCD}, the left-hand side of \eqref{cont01} is negative. Then taking limits as $t\to t_0$ and by using dominated convergence theorem we find
\begin{align}\label{cont02}
0\ge &-\int^{x_2}_{x_1}{u(t_0,x)\left(\L_X^*\varphi(x)-r\varphi(x)\right)dx}-\int^{x_2}_{x_1}{H(t_0,x)
\varphi(x)dx}\\[+3pt]
=&-\int^{x_2}_{x_1}{H(t_0,x)
\varphi(x)dx}\nonumber
\end{align}
where we have used that $u(t_0,x)=0$ for $x\in(x_1,x_2)$ by \eqref{cont00}. We now observe that $H(t_0,x)<-\ell$ for $x\in(x_1,x_2)$ and a suitable $\ell>0$ by \eqref{ltsf-3}, therefore \eqref{cont02} leads to a contradiction and it must be $b^{(2)}(t_0-)=b^{(2)}(t_0)$.
\vs{+6pt}

\emph{Step 3}. To prove that $c^{(2)}$ is left-continuous we can use arguments that follow the very same lines as those in step 2 above and therefore we omit them for brevity.
\end{proof}

%%%%%%%%%%%%%%%%%%%%%%%%%%%%%%%%%%%%%%%%%%%%%%%%%%%%%%%%%%%%%%%%%%%%%%%%%%%%%%%%%
\subsubsection{The EEP representation of the option's value and equations for the boundaries}\label{subs:EEP}

Finally we are able to find an early-exercise premium (EEP) representation for $V^{(2)}$ of problem \eqref{3.1} and a system of coupled integral equations for the free-boundaries $b^{(2)}$ and $c^{(2)}$. The following functions will be needed in the next theorem:
\begin{align}\hs{-2pc}
\label{defJ}&J(t,x):=\EE_x G(T_\delta,X_{T_\delta-t}),\\[+3pt]
\label{defK}&K(b^{(1)},b^{(2)},c^{(2)}\,;x,t,s):=\PP_x(X_s\le b^{(2)}(t\p s))\\[+3pt]
&\hspace{+120pt}\p e^{-r\delta}\PP_x\big(X_s \le b^{(2)}(t\p s),X_{s+\delta} \le b^{(1)}(t\p s\p\delta)\big)\nonumber\\
&\hspace{+120pt}\p e^{-r\delta}\PP_x\big(X_s \ge c^{(2)}(t\p s),X_{s+\delta} \le b^{(1)}(t\p s\p\delta)\big).\nonumber
\end{align}
\begin{theorem}\label{thm:eep2}
The value function $V^{(2)}$ of \eqref{3.1} has the following representation
\begin{align}\label{3.45} \hs{-2pc}
V^{(2)}(t,x)=\;e^{-r(T_\delta-t)}J(t,x)+rK\int_0^{T_\delta-t} e^{-rs}K(b^{(1)},b^{(2)},c^{(2)}\,;x,t,s) ds
\end{align}
for $t\in[0,T_\delta]$ and $x\in (0,\infty)$. The optimal stopping boundaries $b^{(2)}$ and $c^{(2)}$ of \eqref{3.38} and \eqref{3.39} are the unique couple of continuous functions solving the system of coupled nonlinear integral equations
\begin{align} %\hs{1pc}
\label{3.46}&G(t,b^{(2)}(t))=\;e^{-r(T_\delta-t)}J(t,b^{(2)}(t))+rK\int_0^{T_\delta-t} e^{-rs}K(b^{(1)},b^{(2)},c^{(2)}\,;b^{(2)}(t),t,s) ds\\[+3pt]
\label{3.47}&G(t,c^{(2)}(t))=\;e^{-r(T_\delta-t)}J(t,c^{(2)}(t))+rK\int_0^{T_\delta-t} e^{-rs}K(b^{(1)},b^{(2)},c^{(2)}\,;c^{(2)}(t),t,s) ds
\end{align}
with $b^{(2)}(T_\delta)=c^{(2)}(T_\delta)=K$ and $b^{(2)}(t)\le K\le c^{(2)}(t)$ for $t\in[0,T_\delta]$.
\end{theorem}
\begin{proof}
\emph{Step 1 - Existence}. Here we prove that $b^{(2)}$ and $c^{(2)}$ solve \eqref{3.46}--\eqref{3.47}. We start by recalling that the following conditions hold:
\emph{(i)} $V^{(2)}$ is $C^{1,2}$ separately in $C^{(2)}$ and $D^{(2)}$ and $V^{(2)}_t\p\L_{X} V^{(2)}\m rV^{(2)}$ is locally bounded on $C^{(2)}\cup D^{(2)}$ (cf.~\eqref{3.31}--\eqref{3.36} and \eqref{ltsf-3});
\emph{(ii)} $b^{(2)}$ and $c^{(2)}$ are of bounded variation due to monotonicity;
\emph{(iii)} $x\mapsto V^{(2)}(t,x)$ is convex (recall proof of Proposition \ref{prop:contV});
\emph{(iv)} $t\mapsto V^{(2)}_x (t,b^{(2)}(t)\pm)$ and $t\mapsto V^{(2)}_x (t,c^{(2)}(t)\pm)$ are continuous for $t\in[0,T_\delta)$ by \eqref{3.34} and \eqref{3.35}. Hence for any $(t,x)\in[0,T_\delta]\times(0,\infty)$ and $s\in[0,T_\delta\m t]$ we can apply the local time-space formula on curves of \cite{Pe-1} to obtain
\begin{align} \label{3.48} \hs{-2pc}
e^{-rs} &V^{(2)}(t\p s,X^x_s)\\
=&V^{(2)}(t,x)+M_u\nonumber\\
 &\p \int_0^{s}
 e^{-ru}(V^{(2)}_t \p\L_X V^{(2)}\m rV^{(2)})(t\p u,X^x_u)
 I\big(X^x_u\neq \{b^{(2)}(t\p u),\, c^{(2)}(t\p u)\}\big)du\nonumber\\
 =\;&V^{(2)}(t,x)+M_u + \int_0^{s}
 e^{-ru}H(t\p u,X^x_u)\Big[I\big(X^x_u < b^{(2)}(t\p u)\big)\p I\big(X^x_u > c^{(2)}(t\p u)\big)\Big]du\nonumber\\
 %&+\int_0^{s}
 %e^{-ru}(G_t \p\L_X G\m rG)(t\p u,X^x_u)I\big(X^x_u > c^{(2)}(t\p u)\big)du\nonumber\\
 =\;&V^{(2)}(t,x)+M_u - rK\int_0^{s}
 e^{-ru}\big(1+f(t\p u,X^x_u)\big)I\big(X^x_u < b^{(2)}(t\p u)\big)du\nonumber\\
 &-rK\int_0^{s}
 e^{-ru}f(t\p u,X^x_u)I\big(X^x_u > c^{(2)}(t\p u)\big)du\nonumber
\end{align}
where we used \eqref{ltsf-3}, \eqref{3.31} and smooth-fit conditions \eqref{3.34}-\eqref{3.35} and where $M=(M_u)_{u\ge 0}$ is a real martingale. Recall that the law of $X^x_u$ is absolutely continuous with respect to the Lebesgue measure for all $u> 0$, then from strong Markov property and \eqref{3.6} we deduce
\begin{align}\label{3.48b}\hs{-2pc}
f(t\p u,X^x_u)I\big(X^x_u < b^{(2)}(t\p u)\big)&=\PP\big(X^x_{u+\delta}\le b^{(1)}(t\p u\p\delta)\big|\cF_u\big)
I\big(X^x_{u}< b^{(2)}(t\p u)\big)\\[+4pt]
&=\PP\big(X^x_{u+\delta}\le b^{(1)}(t\p u\p\delta), X^x_{u}\le b^{(2)}(t\p u)\big|\cF_u\big)\nonumber
\end{align}
and analogously we have that
\begin{align}\label{3.48c}%\hs{2pc}
f(t\p u,X^x_u)I\big(X^x_u > c^{(2)}(t\p u)\big)=\PP\big(X^x_{u+\delta}\le b^{(1)}(t\p u\p\delta), X^x_{u}\ge c^{(2)}(t\p u)\big|\cF_u\big).
\end{align}
In \eqref{3.48} we let $s=T_\delta - t$, take the expectation $\EE$, use \eqref{3.48b}-\eqref{3.48c} and the optional sampling theorem for $M$, then after rearranging terms and noting that $V^{(2)}(T_\delta,x)=G(T_\delta,x)$ for all $x> 0$, we get \eqref{3.45}.
The system of coupled integral equations \eqref{3.46}-\eqref{3.47} is obtained by simply putting $x=b^{(2)}(t)$ and
$x=c^{(2)}(t)$ into \eqref{3.45} and using \eqref{3.32}-\eqref{3.33}.
\vs{6pt}

%%%%%%%%%%%%%%%%%%%%%%%%%%%%%%%%%%%%%%%%%%%%%%%%%%%%%%%%%%%%%%%%%%%%%%%%%%%%%%%
%%% Figure 1 %%%
%%%%%%%%%%%%%%%%%%%%%%%%%%%%%%%%%%%%%%%%%%%%%%%%%%%%%%%%%%%%%%%%%%%%%%%%%%%%%%%
\begin{figure}[t]
\label{fig:1}
\begin{center}
\includegraphics[scale=0.6]{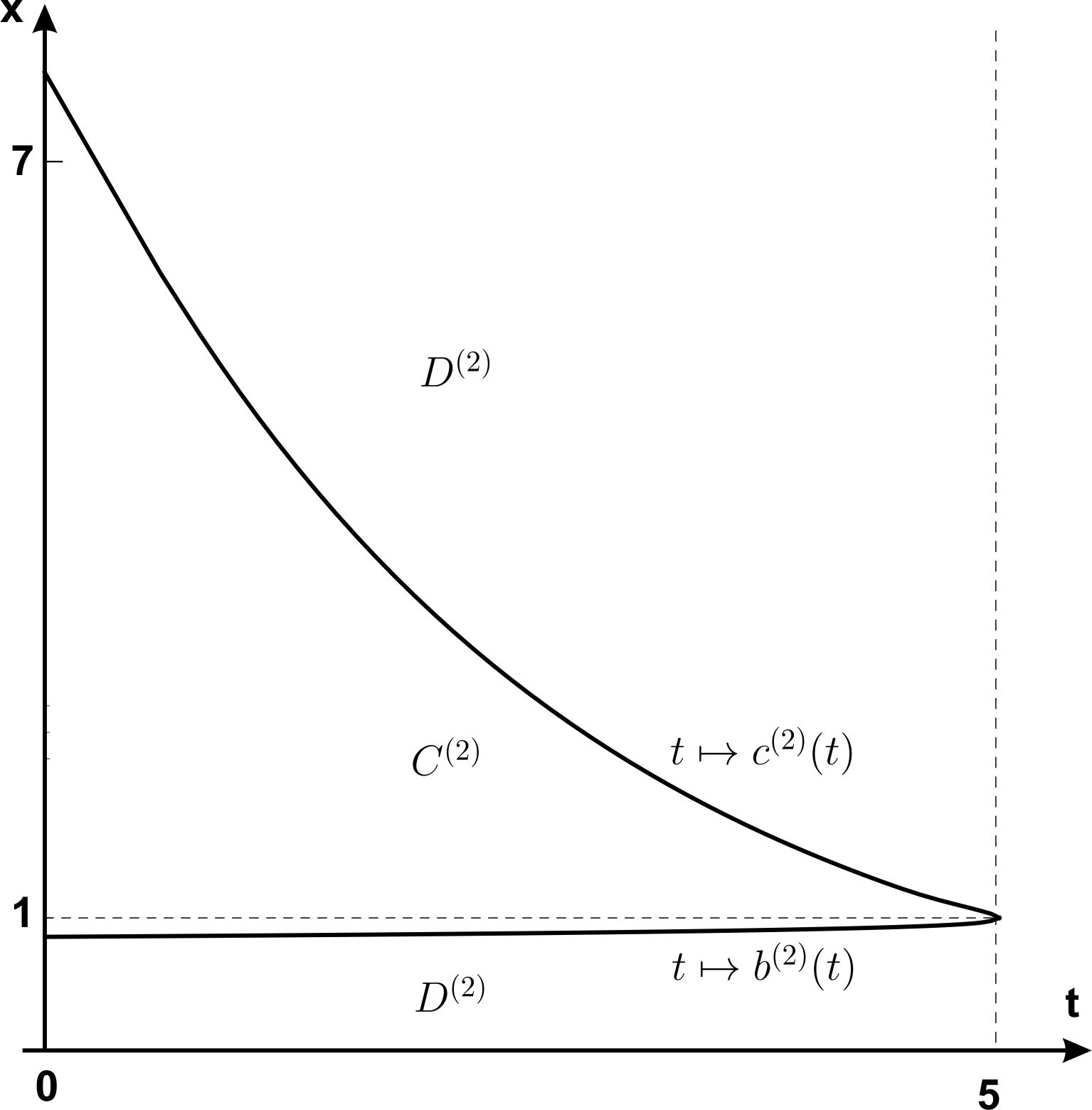}
\end{center}

{\par \leftskip=1.6cm \rightskip=1.6cm \small \ni \vs{-10pt}

\textbf{Figure 1.}
A computer drawing of the optimal exercise boundaries
$t\mapsto b^{(2)}(t)$ and $t\mapsto c^{(2)}(t)$ in the case $K=1$, $r=0.05$ (annual), $\sigma=0.2$ (annual), $T=6$ months, $\delta=1$ month.
\par} \vs{10pt}

\end{figure}

\emph{Step 2 - Uniqueness}. Now we show uniqueness of the solution pair
to the system \eqref{3.46}-\eqref{3.47}. 
The proof is divided in five steps and it is based on arguments similar to those employed in \cite{DuTPe} and originally derived in \cite{Pe-0}.
\vs{6pt}

\emph{Step 2.1}. Let $b:[0,T_\delta]\rightarrow (0,\infty)$ and $c:[0,T_\delta]\rightarrow (0,\infty)$ be another solution pair to the system
\eqref{3.46}-\eqref{3.47} such that $b$ and $c$ are continuous and $b(t)\le K\le c(t)$ for all $t\in[0,T_\delta]$. We will show that these $b$ and $c$ must be equal to the optimal stopping boundaries $b^{(2)}$ and $c^{(2)}$, respectively.

We define a continuous function $U^{b,c}:[0,T_\delta]\times(0,\infty)\rightarrow \R$ by
\begin{align*} \hs{-2pc}%\label{3.49}
U^{b,c}(t,x):=\;&e^{-r(T_\delta-t)}\EE G(T_\delta,X^x_{T_\delta-t})\\
&-\EE\int_0^{T_\delta-t} e^{-ru}H(t\p u,X^x_u)I(X^x_u\le b(t\p u)\;\text{or}\;X^x_u\ge c(t\p u))du.\nonumber
\end{align*}
%for $(t,x)\in[0,T_\delta]\times(0,\infty)$. 
Observe that since $b$ and $c$ solve the system
\eqref{3.46}-\eqref{3.47} then $U^{b,c}(t,b(t))=G(t,b(t))$ and $U^{b,c}(t,c(t))=G(t,c(t))$ for all $t\in [0,T_\delta]$. Notice also that the Markov property of $X$ gives
\begin{align}\label{3.50} \hs{-2pc}
e^{-rs}U^{b,c}(t\p s,X^x_s)&-\int_0^{s} e^{-ru}H(t\p u,X^x_u)I(X^x_u\le b(t\p u)\;\text{or}\;X^x_u\ge c(t\p u))du\\[+4pt]
=&\;U^{b,c}(t,x)+N_s\qquad\PP-\text{a.s.}\nonumber
\end{align}
for $s\in[0,T_\delta-t]$ and with $(N_s)_{0\le s\le T_\delta -t}$ a $\PP$-martingale.
\vspace{+6pt}

\emph{Step 2.2}. We now show that
$U^{b,c}(t,x)=G(t,x)$ for $x\in(0,b(t)]\cup [c(t),\infty)$ and $t\in [0,T_\delta]$. For $x\in(0,b(t)]\cup [c(t),\infty)$ with $t\in[0,T_\delta]$ given and fixed, consider the stopping time
 \begin{align} \label{3.52} \hs{-2pc}
\sigma_{b,c}=\sigma_{b,c}(t,x)=\inf\ \{\ 0\leq s\leq T_{\delta}\m t:b(t\p s)\le X^x_{s}\le c(t\p s)\ \}.
 \end{align}
Using that $U^{b,c}(t,b(t))=G(t,b(t))$ and $U^{b,c}(t,c(t))=G(t,c(t))$ for all $t\in [0,T_\delta)$ and $U^{b,c}(T_\delta,x)=G(T_\delta,x)$ for all $x> 0$, we get $U^{b,c}(t\p\sigma_{b,c},X^x_{\sigma_{b,c}})=G(t\p\sigma_{b,c},X^x_{\sigma_{b,c}})$ $\PP$-a.s.~by continuity of $U^{b,c}$. Hence
 from \eqref{3.8} and \eqref{3.50} using the optional sampling theorem and noting that $L^K_u(X^x)=0$ for $u\le\sigma_{b,c}$ we find
 \begin{align*} \hs{-2pc}%\label{3.53}
U^{b,c}(t,x)=\;&\EE e^{-r\sigma_{b,c}}U^{b,c}(t\p \sigma_{b,c},X^x_{\sigma_{b,c}})\\
&-\EE\int_0^{\sigma_{b,c}} e^{-ru}H(t\p u,X^x_u)I(X^x_u\le b(t\p u)\;\text{or}\;X^x_u\ge c(t\p u))du\nonumber\\
=\;&\EE e^{-r\sigma_{b,c}}G(t\p \sigma_{b,c},X^x_{\sigma_{b,c}})%\nonumber\\
-\EE\int_0^{\sigma_{b,c}} e^{-ru}H(t\p u,X^x_u)du=G(t,x)\nonumber
\end{align*}
 since $X^x_u \in (0,b(t\p u))\cup (c(t\p u),\infty)$ for all $u\in [0,\sigma_{b,c})$.
 \vs{6pt}

\emph{Step 2.3}. Next we prove that $U^{b,c}(t,x)\le V^{(2)}(t,x)$ for all $(t,x)\in[0,T_\delta]\times(0,\infty)$.
 For this consider the stopping time
 \begin{align*} \hs{1pc}%\label{3.54}
\tau_{b,c}=\tau_{b,c}(t,x)=\inf\ \{\ 0\leq s\leq T_{\delta}\m t:X^x_s\le b(t\p s)\;\text{or}\;X^x_s\ge c(t\p s)\ \}
 \end{align*}
with $(t,x)\in[0,T_\delta]\times(0,\infty)$ given and fixed. Again arguments as those
following \eqref{3.52} above show that $U^{b,c}(t\p\tau_{b,c},X^x_{\tau_{b,c}})=G(t\p\tau_{b,c},X^x_{\tau_{b,c}})$ $\PP$-a.s. Then taking $s=\tau_{b,c}$ in \eqref{3.50} and using the optional sampling theorem, we get
 \begin{align*}\hs{1pc}%\label{3.55}
U^{b,c}(t,x)=\EE e^{-r\tau_{b,c}}U^{b,c}(t\p \tau_{b,c},X^x_{\tau_{b,c}})=
\EE e^{-r\tau_{b,c}}G(&t\p \tau_{b,c},X^x_{\tau_{b,c}})\le V^{(2)}(t,x).
\end{align*}
% \vs{6pt}

\emph{Step 2.4}. In order to compare the couples $(b,c)$ and $(b^{(2)}, c^{(2)})$ we initially prove that $b(t)\ge b^{(2)}(t)$ and $c(t)\le c^{(2)}(t)$ for $t\in[0,T_\delta]$. For this, suppose that there exists $t\in[0,T_\delta)$ such that $c(t)> c^{(2)}(t)$, take a point $x\in [c(t),\infty)$ and consider the stopping time
  \begin{align*}\hs{1pc} %\label{3.56}
\sigma=\sigma(t,x)=\inf\ \{\ 0\leq s\leq T_{\delta}\m t:b^{(2)}(t\p s)\le X^x_{s}\le c^{(2)}(t\p s)\ \}.
 \end{align*}
Setting $s=\sigma$ in \eqref{3.48} and \eqref{3.50} and using the optional sampling theorem, we get
\begin{align} \label{3.57} \hs{1pc}
&\EE e^{-r\sigma} V^{(2)}(t\p \sigma,X^x_\sigma)=V^{(2)}(t,x) +\EE\int_0^\sigma e^{-ru}H(t\p u,X^x_u)
du\\
\label{3.58}&\EE e^{-r\sigma}U^{b,c}(t\p \sigma,X^x_\sigma)=U^{b,c}(t,x)\\
&\qquad\qquad+\EE\int_0^{\sigma} e^{-ru}H(t\p u,X^x_u)I\big(X^x_u \le b(t\p u)\;\text{or}\;X^x_u\ge c(t\p u)\big)du\nonumber.
\end{align}
Since $U^{b,c}\le V^{(2)}$ and $V^{(2)}(t,x)=U^{b,c}(t,x)=G(t,x)$ for
$x\in[c(t),\infty)$, it follows by subtracting \eqref{3.58} from \eqref{3.57} that
\begin{align} \label{3.59} \hs{1pc}
\EE\int_0^{\sigma} e^{-ru}H(t\p u,X^x_u)I\big(b(t\p u)\le X^x_u\le c(t\p u)\big)du\ge0.
\end{align}
The function $H$ is always strictly negative and by the continuity of $c^{(2)}$
and $c$ it must be $\PP(\sigma(t,x)>0)=1$, hence \eqref{3.59} leads to a contradiction and we can conclude that $c(t)\le c^{(2)}(t)$ for all $t\in [0,T_\delta]$. Arguing in a similar way one can also derive that
$b(t)\ge b^{(2)}(t)$ for all $t\in [0,T_\delta]$ as claimed.
\vs{6pt}

\emph{Step 2.5}. To conclude the proof we show that $b=b^{(2)}$ and $c=c^{(2)}$ on $[0,T_\delta]$.
For that, let us assume that there exists $t\in[0,T_\delta)$ such that $b(t)>b^{(2)}(t)$ or $c(t)<c^{(2)}(t)$. Choose an arbitrary point $x\in (b^{(2)}(t),b(t))$ or alternatively $x\in(c(t),c^{(2)}(t))$ and
consider the optimal stopping time $\tau^*$ of \eqref{3.11} with $D^{(2)}$ as in \eqref{3.39}. Take $s=\tau^*$ in \eqref{3.48} and \eqref{3.50} and use the optional sampling theorem to get
 \begin{align} \label{3.60} \hs{1pc}
&\EE e^{-r\tau^*} G(t\p \tau^*,X^x_{\tau^*})=V^{(2)}(t,x)\\
\label{3.61}&\EE e^{-r{\tau^*}}G(t\p {\tau^*},X^x_{\tau^*})=U^{b,c}(t,x)\\
&\qquad\qquad+\EE\int_0^{{\tau^*}} e^{-ru}H(t\p u,X^x_u)I\big(X^x_u \le b(t\p u)\;\text{or}\;X^x_u\ge c(t\p u)\big)du\nonumber
\end{align}
where we use that $V^{(2)}(t\p \tau^*,X^x_{\tau^*})=G(t\p \tau^*,X^x_{\tau^*})=U^{b,c}(t\p \tau^*,X^x_{\tau^*})$ $\PP$-a.s.~upon recalling that $b\ge b^{(2)}$ and $c\le c^{(2)}$, and $U^{b,c}=G$ either below $b$ and above $c$ (cf.~step 2.2 above) or at $T_\delta$. Since $U^{b,c}\le V^{(2)}$ then subtracting \eqref{3.60} from \eqref{3.61} we get
 \begin{align*}  \hs{2pc}%\label{3.62}
\EE\int_0^{{\tau^*}} e^{-ru}H(t\p u,X^x_u)I\big(X^x_u \le b(t\p u)\;\text{or}\;X^x_u\ge c(t\p u)\big)du\ge 0.
\end{align*}
Again we recall that $H$ is always strictly negative and by continuity of
$b^{(2)}$, $c^{(2)}$, $b$ and $c$ we have $\PP(\tau^*(t,x)>0)=1$ and the process $(X^{x}_u)_{u\in[0,T_\delta-t]}$ spends a strictly positive amount of time either below $b(t\p\, \cdot \,)$ if it starts from $x\in\big(b^{(2)}(t),b(t)\big)$ or above $c(t\p \, \cdot \,)$ if it starts from $x\in\big(c(t),c^{(2)}(t)\big)$, with probability one. Therefore we reach a contradiction unless $b=b^{(2)}$ and $c=c^{(2)}$.
\end{proof}

It is worth observing that the pricing formula \eqref{3.45} is consistent with the economic intuition behind the structure of the swing contract and it includes a European part plus three integral terms accounting for the early exercise premia. The first of such terms is similar to the one appearing in the American put price formula and it represents the value produced by a single exercise when the put payoff is strictly positive.
The second and third terms instead are related to the extra value produced by the multiple exercise opportunity. These premia are weighted with the discounted probability of the price process falling below $b^{(1)}$ after the refracting period has elapsed and they account for both the cases when the first right is exercised below the strike $K$ or above it, respectively.

Notice that, as anticipated in Remark \ref{rem:optexn.1}, if $r = 0$ the early exercise premia disappear. In that case there is no time value of money and the swing contract is equivalent to a portfolio of two European options: one with maturity at time $T_\delta$ and the other at time $T$. In Figure $2$ we show optimal boundaries evaluated numerically for some values of $r$ and observe that the stopping region increases as the interest rate increases (i.e.~the continuation set shrinks as $r$ increases).

It is well known that for $\delta=0$ the value of the swing contract with $n$ rights equals the value of a portfolio of $n$ American put options with maturity at $T$. In that setting it is optimal to exercise all rights at the same time as soon as $X_t$ falls below $b^{(1)}$. Intuitively we expect that as $\delta\to 0$ and $T_\delta\to T$, the lower boundary $b^{(2)}$ tends to the boundary of the American put $b^{(1)}$ while the upper boundary $c^{(2)}$ increases and becomes steeper near the maturity, eventually converging to the vertical half line $\{T\}\times(K,\infty)$ in the limit (see Figure 3 for numerical illustrations of this observation).

\begin{remark}\label{rem:numerics}
We compute the optimal boundaries of Figures $1$, $2$, $3$ and $4$ by solving numerically the integral equations \eqref{3.46} and \eqref{3.47} (see also \eqref{inteq-n0} and \eqref{inteq-n1} in the next section). We use a backwards scheme based on a discretisation of the integrals with respect to time. This is a standard method for this kind of equations and a more detailed description can be found for example in Remark 4.2 of \cite{DuTPe}. Note that in order to implement the algorithm it is crucial to know the values of $c^{(2)}(T_\delta)$ and $b^{(2)}(T_\delta)$.
\end{remark}

%%%%%%%%%%%%%%%%%%%%%%%%%%%%%%%%%%%%%%%%%%%%%%%%%%%%%%%%%%%%%%%%%%%%%%%%%%%%%%%
%%% Figure 2 %%%
%%%%%%%%%%%%%%%%%%%%%%%%%%%%%%%%%%%%%%%%%%%%%%%%%%%%%%%%%%%%%%%%%%%%%%%%%%%%%%%
\begin{figure}[t]
\label{fig:2}
\begin{center}
\includegraphics[scale=0.5]{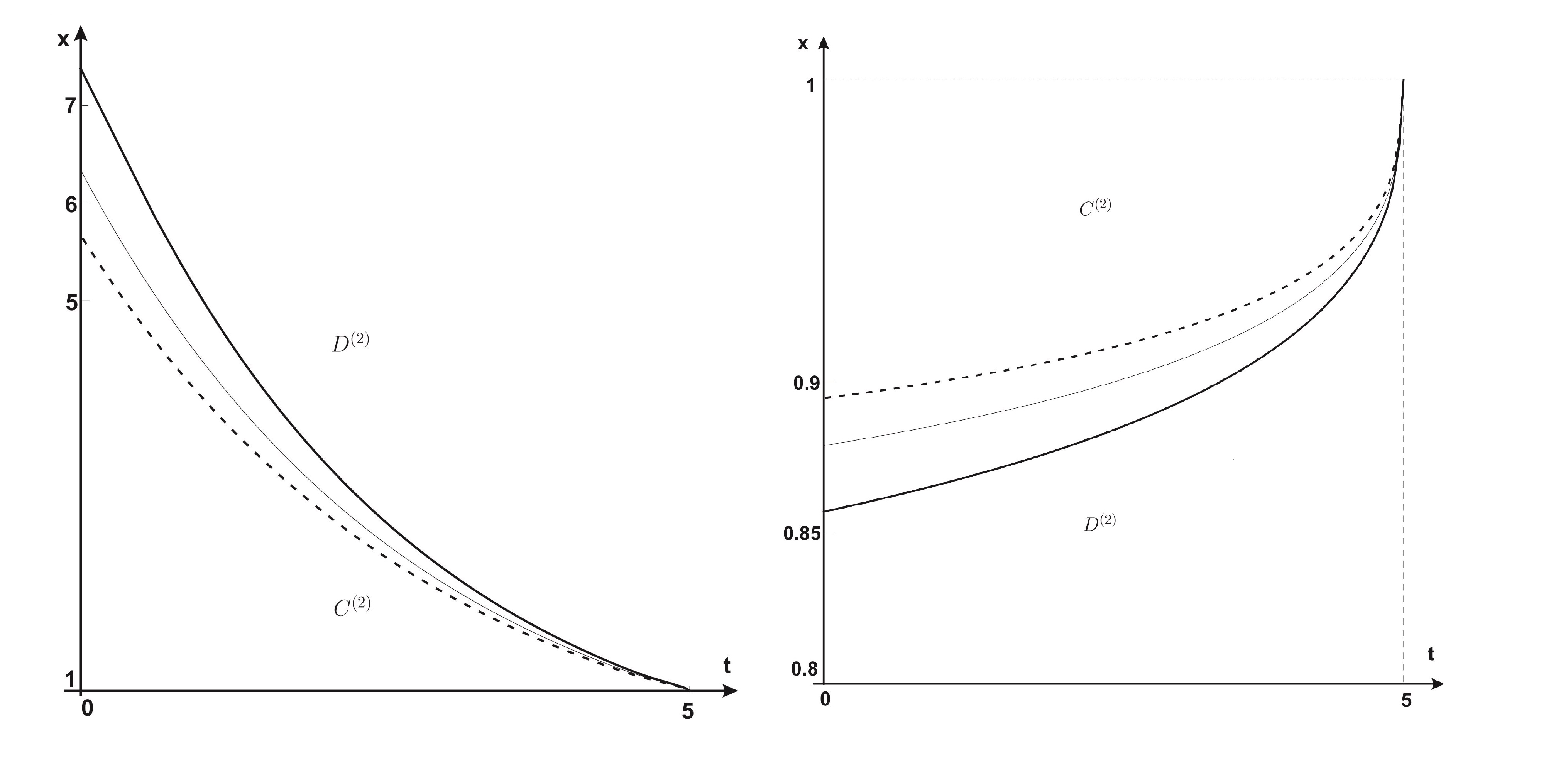}
\end{center}

{\par \leftskip=1.6cm \rightskip=1.6cm \small \ni \vs{-10pt}

\textbf{Figure 2.} Computer drawings show how the upper optimal exercise boundary
$t\mapsto c^{(2)}(t)$ (on the left) and the lower optimal exercise boundary $t\mapsto b^{(2)}(t)$ (on the right)
change as one varies the annual interest rate $r$. The set of parameters is $K=1$, $\sigma=0.4$ (annual), $T=6$ months, $\delta=1$ month
and the boundaries refer to the following values of the interest rate: $r=0.05$ (bold line); $r=0.075$ (thin line); $r=0.1$ (dashed line). The stopping region is increasing in $r$. Note that we use different scales on the vertical axes.

\par} \vs{10pt}

\end{figure}

%%%%%%%%%%%%%%%%%%%%%%%%%%%%%%%%%%%%%%%%%%%%%%%%%%%%%%%%%%%%%%%%%%%%%%%%%%%%%%%
%%% Figure 3 %%%
%%%%%%%%%%%%%%%%%%%%%%%%%%%%%%%%%%%%%%%%%%%%%%%%%%%%%%%%%%%%%%%%%%%%%%%%%%%%%%%
\begin{figure}[t]
%\label{fig:2}
\begin{center}
\includegraphics[scale=0.4]{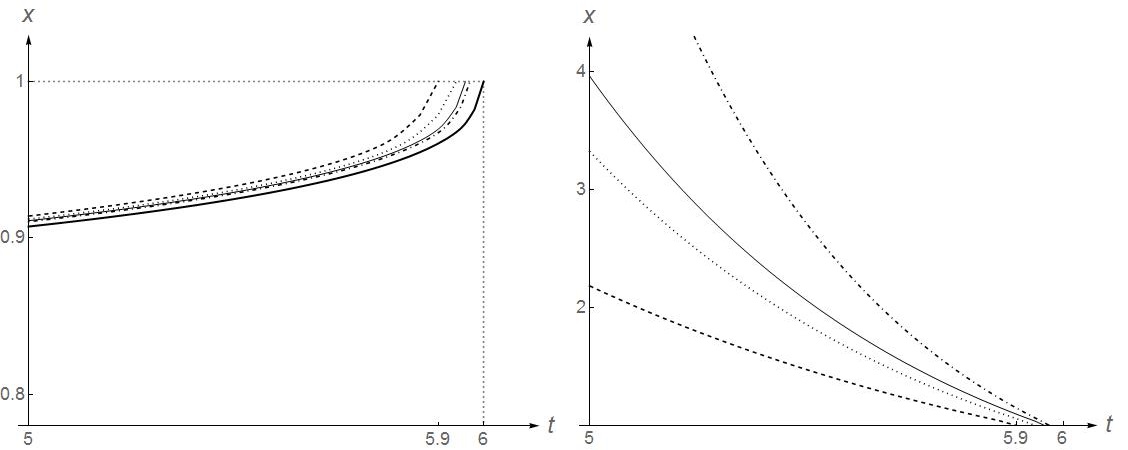}
\end{center}

{\par \leftskip=1.6cm \rightskip=1.6cm \small \ni \vs{-10pt}

\textbf{Figure 3.} Computer drawings show how the optimal exercise boundaries
$t\mapsto b^{(2)}(t)$ and $t\mapsto c^{(2)}(t)$
change as the refracting period $\delta$ goes to 0. The set of parameters is $K=1$, $\sigma=0.4$ (annual), $T=6$ months, $r=0.05$ (annual)
and the boundaries refer to the following values of the refracting period: $\delta=0.1$ (dashed); $\delta=0.06$ (dotted); $\delta=0.04$ (thin); $\delta=0.03$ (dash-dotted). The thick line represents the American put option boundary $t\mapsto b^{(1)}(t)$.
\par} \vs{10pt}

\end{figure}
%%%%%%%%%%%%%%%%%%%%%%%%%%%%%%%%%%%%%%%%%%%%%%%%%%%%%%%%%%%%%%%%%%%%%%%%%%
\subsection{Analysis of the swing option with arbitrary many rights}\label{sec:gen}

In this section we complete our study of the multiple optimal stopping problem \eqref{eq:V1bis} by dealing with the general case of $n$ admissible stopping times. The results follow by induction and we will only sketch their proofs as they are obtained by repeating step by step arguments as those presented in Section \ref{sec:n=2}. Let us start by introducing some notation. For $n\in\mathbb{N}$, $n\ge 2$ we denote $C^{(n)}$ and $D^{(n)}$ the continuation and stopping region, respectively, of the problem with value function $V^{(n)}$ (cf.~\eqref{def:OS00}). Similarly we denote their $t$-sections by $C^{(n)}_t$ and $D^{(n)}_t$ and to simplify notation we set $T^{(n)}_\delta=T\m n\delta$.

From now on we fix $n\ge 2$ and make some assumptions that will be needed to obtain properties of $G^{(n+1)}$, $V^{(n+1)}$ and the relative optimal boundaries. Notice that each one of the following assumptions hold for $n=2$.
\begin{ass}\label{ass:induction}
%For some fixed $n\ge 2$
For $j\in\{2,3,\ldots n\}$ and $t\in[0,T^{(j-1)}_\delta]$ one has $D^{(j)}_t=(0,b^{(j)}(t)]\cup[c^{(j)}(t),\infty)$ where
\begin{itemize}
\item[  i)] $t\mapsto b^{(j)}(t)$ is continuous, bounded and increasing with $b^{(j)}\big(T^{(j-1)}_\delta\big)=K$,
\item[ ii)] $t\mapsto c^{(j)}(t)$ is continuous, bounded and decreasing with $c^{(j)}\big(T^{(j-1)}_\delta\big)=K$,
\item[iii)] $b^{(j-1)}(t)\le b^{(j)}(t)< K< c^{(j)}(t)\le c^{(j-1)}(t)$ for $t\in[0,T^{(j-1)}_\delta)$, with the convention $c^{(1)}\equiv+\infty$.
\end{itemize}
\end{ass}
%In Section \ref{sec:n=2} we have shown indeed that Assumption \ref{ass:induction} holds for $n=2$.
Now for $1\le j\le n-1$ we also define the random variables
\begin{align}
\label{eq:defInj} I^{(n)}_j(t,s):=I\big(X_s&\in D^{(n)}_{t+s},X_{s+\delta}\in D^{(n-1)}_{t+s+\delta},\ldots\\
&\ldots,X_{s+(j-1)\delta}\in D^{(n-(j-1))}_{t+s+(j-1)\delta}, X_{s+j\delta}<b^{(n-j)}(t\p s\p j\delta)\big)\,\nonumber\\[+3pt]
\label{eq:defIn0}I^{(n)}_0(t,s):=I\big(X_s&< b^{(n)}(t\p s)\big)\
\end{align}
whose expected values are instead denoted by
\begin{align}
\label{eq:defpnj} p^{(n)}_j(t,x,s):=\EE_x \big[I^{(n)}_j(t,s)\big]\quad\text{and}\quad p^{(n)}_0(t,x,s):=\EE_x\big[I^{(n)}_0(t,s)\big].
\end{align}
Under Assumption \ref{ass:induction} one has
\begin{align}\label{eq:probs}\hs{+4pc}
p^{(n)}_j(t,x,s)-p^{(n-1)}_j(t,x,s)\ge 0
\end{align}
for $t\in[0,T^{(n-1)}_\delta]$, $(x,s)\in(0,\infty)\times[0,T^{(n-1)}_\delta\m\,t]$ and $j=0,\ldots,n-2$ since $D^{(1)}\subseteq D^{(2)}\subseteq\ldots\subseteq D^{(n)}$.
Let us recall definition \eqref{def:payoff-n01} in order to introduce the next
\begin{ass}\label{ass:induction2}
It holds $G^{(n)}\in C^{1,2}$ in $(0,T^{(n-1)}_\delta)\times[(0,K)\cup(K,\infty)]$ with
\begin{align}\label{LGn}%\hs{-1pc}
\big(G^{(n)}_t\p\L_XG^{(n)}\m rG^{(n)}\big)(t,x)
=-rK\big(I&(x<K)+\sum^{n-2}_{j=0}e^{-r(j+1)\delta}p^{(n-1)}_j(t,x,\delta)\big)
\end{align}
for $t\in(0,T^{(n- 1)}_\delta)$ and $x\in(0,K)\cup(K,\infty)$.
Moreover $V^{(n)}$ is continuous on $[0,T^{(n-1)}_\delta]\times(0,\infty)$, $V^{(n)}\in C^{1,2}$ in $C^{(n)}$, and it solves
\begin{align}\label{eq:Vn}\hs{+4pc}
V^{(n)}_t+\L_XV^{(n)}-rV^{(n)}=0\quad\text{in $C^{(n)}$}.
\end{align}
Finally, for $s\in[0,T^{(n-1)}_\delta-t]$ and $x\in(0,\infty)$ it holds $\PP_x$-a.s.
\begin{align}\label{itoVn}
e^{-rs}V^{(n)}(t\p s,X_s)=V^{(n)}(t,x)\m rK\sum^{n-1}_{j=0}\int_0^s{e^{-r(u+j\delta)}\EE_{x}\big[I^{(n)}_j(t,u)\big|\cF_u\big]du}\p M^{(n)}_{t+s}
\end{align}
where $(M^{(n)}_t)_{t}$ is a martingale.
\end{ass}
Notice that for $n=2$ \eqref{LGn} is equivalent to \eqref{ltsf-3} whereas \eqref{itoVn} is equivalent to \eqref{3.48}.
\begin{prop}\label{prop:LGn+1}
Under Assumptions \ref{ass:induction} and \ref{ass:induction2} the equation \eqref{LGn} also holds with $n$ replaced by $n+1$.
\end{prop}
\begin{proof}
Since
\begin{align*}%\label{Gn00}
G^{(n+1)}(t,x)=&(K-x)^++R^{(n+1)}(t,x)
\end{align*}
with $R^{(n+1)}(t,x)=\EE e^{-r\delta} V^{(n)}(t\p\delta, X^x_\delta)$, it is then sufficient to prove that
\begin{align*}
Y_s:=e^{-rs}R^{(n+1)}(t\p s,X_s)+rK\sum^{n-1}_{j=0}\int_0^s{e^{-r(u+(j+1)\delta)}p^{(n)}_j(t\p u,X_u,\delta)du}
\end{align*}
is a continuous martingale. Then we can argue as in the proof of Proposition \ref{rem:reguC^2} to conclude that $R^{(n+1)}$ is $C^{1,2}$ and
\begin{align*}
\big(R^{(n+1)}_t+\L_XR^{(n+1)}-rR^{(n+1)}\big)(t,x)=-rK\sum^{n-1}_{j=0}e^{-r(j+1)\delta}p^{(n)}_j(t,x,\delta).
\end{align*}

Continuity of $s\mapsto Y_s$ follows from continuity and boundedness of $V^{(n)}$ and $p^{(n)}_j$ for all $j$'s. Markov property and \eqref{itoVn} (see also the proof of Proposition \ref{rem:reguC^2} for the details) give
\begin{align*}
\EE_xe^{-rs}R^{(n+1)}(t\p s,X_s)=&R^{(n+1)}(t,x)-rK\sum^{n-1}_{j=0}\EE_x\int_0^se^{-r(u+(j+1)\delta)}I^{(n)}_j(t,u+\delta)du\\
=&R^{(n+1)}(t,x)-rK\sum^{n-1}_{j=0}\EE_x\int_0^se^{-r(u+(j+1)\delta)}\EE_{X_u} \big[I^{(n)}_j(t+u,\delta)\big]du\\
=&R^{(n+1)}(t,x)-rK\sum^{n-1}_{j=0}\EE_x\int_0^se^{-r(u+(j+1)\delta)}p^{(n)}_j(t+u,X_u,\delta)du
\end{align*}
where we have used that $\EE_{x} \big[I^{(n)}_j(t,u+\delta)\big]=\EE_x \EE_{X_u} \big[I^{(n)}_j(t+u,\delta)\big]=\EE_x p^{(n)}_j(t+u,X_u,\delta)$ by \eqref{eq:defInj}.

\end{proof}

We now define
\begin{align}\label{eq:defHn}\hs{-2pc}
H^{(n)}(t,x):= \big(G^{(n)}_t+&\L_XG^{(n)}-rG^{(n)}\big)(t,x)\quad\text{for $t\in(0,T^{(n)}_\delta)$, $x\in(0,K)\cup(K,\infty)$}
\end{align}
and observe that under Assumption \ref{ass:induction} the map $t\mapsto H^{(n)}(t,x)$ is decreasing for all $x>0$. This was also the case for $H$ in \eqref{ltsf-3} and it was the key property needed to prove most of our results in Section \ref{sec:n=2}. We are now ready to provide the EEP representation formula of $V^{(n)}$ for $n>2$ and to characterise the corresponding stopping sets $D^{(n)}$.
\begin{theorem}\label{thm:bdr-n}
For fixed $n\ge 2$ let Assumptions \ref{ass:induction} and \ref{ass:induction2} hold true. Then the same assumptions hold for $n+1$ and, for $t\in[0,T^{(n)}_\delta]$ and $x\in (0,\infty)$, the value function $V^{(n+1)}$ of \eqref{def:OS00} has the following representation
\begin{align}\label{eep-n}\hs{-2pc}
V^{(n+1)}(t,x)=e^{-r(T^{(n)}_\delta-t)}J^{(n+1)}(t,x)+rK\sum^{n}_{j=0}\int_0^{T^{(n)}_\delta-t}{e^{-r(u+j\delta)}p^{(n+1)}_j(t,x,u)du}
\end{align}
with
\begin{align*}
J^{(n+1)}(t,x):=\EE_x \Big[ G^{(n+1)}\Big(T^{(n)}_\delta,X_{T^{(n)}_\delta-t}\Big)\Big].
\end{align*}
\end{theorem}

%%%%%%%%%%%%%%%%%%%%%%%%%%%%%%%%%%%%%%%%%%%%%%%%%%%%%%%%%%%%%%%%%%%%%%%%%%%%%%%
%%% Figure 3 %%%
%%%%%%%%%%%%%%%%%%%%%%%%%%%%%%%%%%%%%%%%%%%%%%%%%%%%%%%%%%%%%%%%%%%%%%%%%%%%%%%
\begin{figure}[t]
\label{fig:3}
\begin{center}
\includegraphics[scale=0.4]{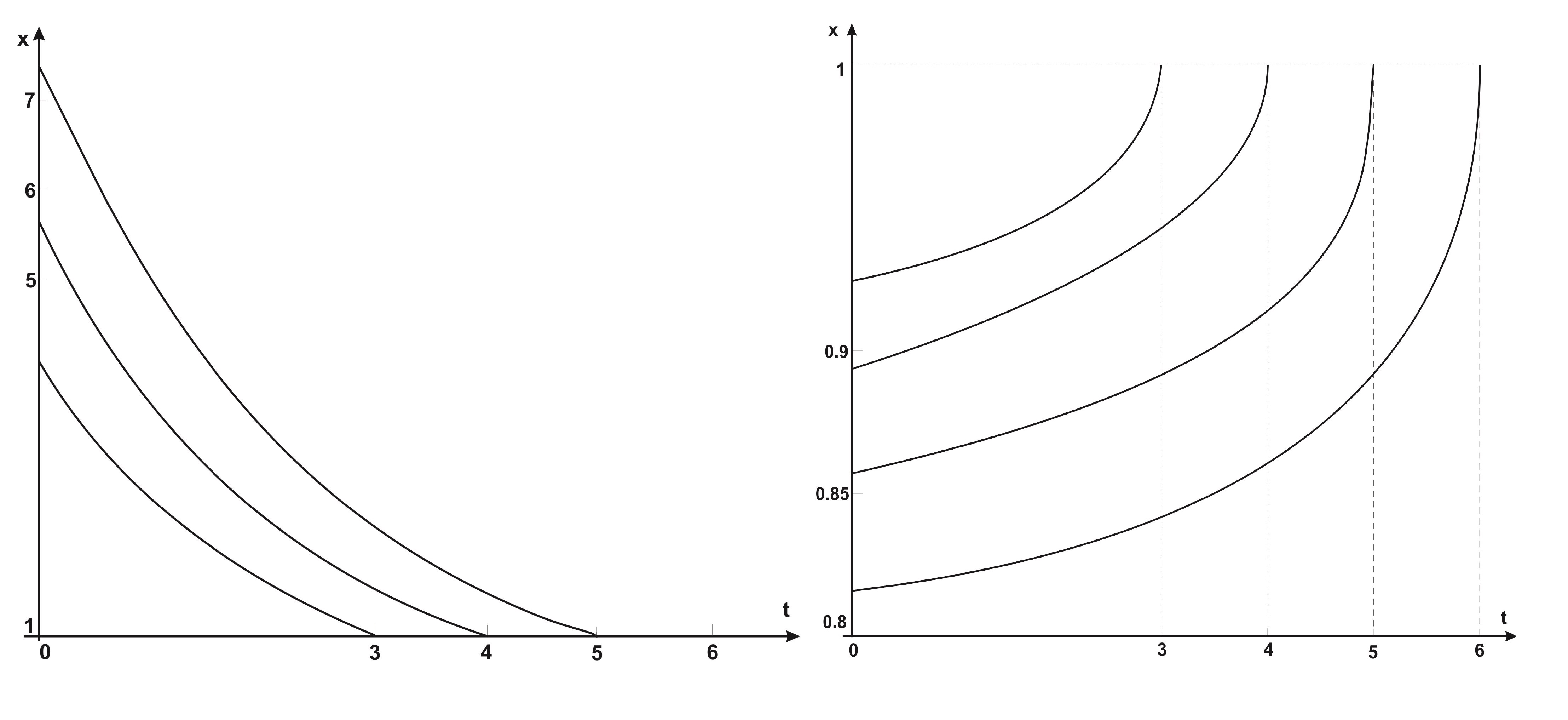}
\end{center}

{\par \leftskip=1.6cm \rightskip=1.6cm \small \ni \vs{-10pt}

\textbf{Figure 4.}
Structure of the upper optimal exercise boundaries
$t\mapsto c^{(n)}(t)$  for $n=2,3,4$ (on the left) and the lower optimal exercise boundaries $t\mapsto b^{(n)}(t)$  for $n=1,2,3,4$ (on the right)
in the case $K=1$, $r=0.05$ (annual), $\sigma=0.2$ (annual), $T=6$ months, $\delta=1$ month. Note that the scales on the vertical axes are different.
\par} \vs{10pt}

\end{figure}

\begin{proof}
By Proposition \ref{prop:LGn+1} we obtain that \eqref{LGn} holds with $n$ replaced by $n+1$ and $H^{(n+1)}$ is well defined (cf.~\eqref{eq:defHn}). Now we repeat step by step (with obvious modifications) the arguments used in Section \ref{sec:n=2} to obtain generalisations of Propositions \ref{prop:contV}, \ref{prop:D2D1}, \ref{prop:incrCD}, \ref{prop:finite-c} and Theorems \ref{thm:b2c2} and \ref{thm:continuity} to the case $n>2$. We observe that some proofs simplify as the generalisation of Proposition \ref{prop:D2D1} (which uses \eqref{eq:probs}) immediately implies finiteness of $c^{(n+1)}$ due to finiteness of $c^{(2)}$ and hence $D^{(n+1)}_t\cap (K,\infty)\neq\emptyset$ for $t\in[0,T^{(n)}_\delta]$. Then for the swing option problem with $n+1$ exercise rights there exist two optimal stopping boundaries $b^{(n+1)}$ and $c^{(n+1)}$ which fulfill Assumption \ref{ass:induction} with $n+1$ instead of $n$ (notice that the proof of Theorem \ref{thm:continuity} does not rely on the smooth-fit property).

It remains to prove that the EEP representation formula for $V^{(n+1)}$ holds. Following the same arguments as in the proof of Proposition \ref{prop:smooth-fit} it is possible to show that $V^{(n+1)}_x(t,\,\cdot\,)$ is continuous across $b^{(n+1)}(t)$ and $c^{(n+1)}(t)$ for all $t\in(0,T^{(n)}_\delta)$. Then $V^{(n+1)}$ solves a free-boundary problem analogous to \eqref{3.31}--\eqref{3.36} but with $V^{(2)}$, $G^{(2)}$, $b^{(2)}$, $c^{(2)}$ and $T_\delta$ replaced by $V^{(n+1)}$, $G^{(n+1)}$, $b^{(n+1)}$, $c^{(n+1)}$ and $T^{(n)}_\delta$ respectively. Now $V^{(n+1)}$, $b^{(n+1)}$ and $c^{(n+1)}$ satisfy all the conditions needed to apply the local time-space formula of \cite{Pe-1} (cf.~also proof of Theorem \ref{thm:eep2} above), hence by using
\begin{align*}%\hs{+2pc}
&\EE_{X_u}\big[I^{(n)}_j(t+u,\delta)\big]\,I(X_u\in D^{(n+1)}_{t+u})=\EE_x\big[I^{(n+1)}_{j+1}(t,u)\big|\cF_u\big]\\[+4pt]
&I(X_u<K)I(X_u\in D^{(n+1)}_{t+u})=I(X_u<b^{(n+1)}(t+u))
\end{align*}
we obtain
\begin{align*}
e^{-r s}&V^{(n+1)}(t\p s,X^x_s)\\
=&V^{(n+1)}(t,x)+\int_0^s{e^{-ru}H^{(n+1)}(t\p u,X^x_u)I(X^x_u\in D^{(n+1)}_{t+u})du}+M^{(n+1)}_{t+s}\\
=&V^{(n+1)}(t,x)\m rK\sum^{n}_{j=0}\int_0^se^{-r(u+j\,\delta)}\EE_x\big[I^{(n+1)}_{j}(t,u)\big|\cF_u\big]du+M^{(n+1)}_{t+s}
\end{align*}
with $M^{(n+1)}$ a martingale. Hence $V^{(n+1)}$ satisfies \eqref{itoVn} and taking $s=T^{(n)}_\delta\m\, t$ and rearranging terms we obtain the EEP representation for the value of the swing option with $n+1$ exercise rights.
\end{proof}

\begin{coroll}
For any $n\ge 2$ Assumptions \ref{ass:induction} and \ref{ass:induction2} hold true and $V^{(n)}$ has the representation \eqref{eep-n} (with $n$ instead of $n+1$).
\end{coroll}
\begin{proof}
From Theorem \ref{thm:bdr-n} we learn that if Assumptions \ref{ass:induction} and \ref{ass:induction2} hold for $V^{(n)}$, $G^{(n)}$, $b^{(n)}$ and $c^{(n)}$ then they also hold for $V^{(n+1)}$, $G^{(n+1)}$, $b^{(n+1)}$ and $c^{(n+1)}$. Since we know from the analysis in Section \ref{sec:n=2} that Assumptions \ref{ass:induction} and \ref{ass:induction2} certainly hold in the case $n=2$, the proof is completed by induction.
\end{proof}
\begin{remark}\label{rem:CarTou}
It is worth observing that in the above corollary we have proven $b^{(j-1)}\le b^{(j)}$ and $c^{(j-1)}\ge c^{(j)}$ for all $j\ge 2$, thus answering positively to a theoretical question that was posed in \cite[Sec.~6.4.3]{Car-Tou08}.
\end{remark}

It is now matter of routine to substitute $b^{(n)}(t)$ and $c^{(n)}(t)$ into \eqref{eep-n} to find the integral equations that characterise the optimal boundaries. Arguments analogous to those employed in step 2 of the proof of Theorem \ref{thm:eep2} allow us to show that $b^{(n)}$ and $c^{(n)}$ uniquely solve such equations. For completeness we provide the theorem but we omit its proof. The following expressions will be needed
\begin{align*}
J^{(n)}(t,x):=\EE_{x}\Big[ G^{(n)}\Big(T^{(n-1)}_\delta,X_{T^{(n-1)}_\delta-t}\Big)\Big]\quad\text{and}\quad\Delta^{(n)}_t:=T^{(n-1)}_\delta-t.
\end{align*}
\begin{theorem}
For all $n\in\mathbb{N}$, $n\ge 2$,  the optimal stopping boundaries $b^{(n)}$ and $c^{(n)}$ of Theorem \ref{thm:bdr-n} are the unique couple of continuous functions solving the system of coupled nonlinear integral equations
\begin{align}%\hs{+4pc}
\label{inteq-n0}&G^{(n)}(t,b^{(n)}(t))=e^{-r\Delta^{(n)}_t}J^{(n)}(t,b^{(n)}(t))\p rK\sum^{n-1}_{j=0}\int_0^{\Delta^{(n)}_t}{e^{-r(u+j\delta)}p^{(n)}_j(t,b^{(n)}(t),u)du} \\[+3pt]
\label{inteq-n1}&G^{(n)}(t,c^{(n)}(t))=e^{-r\Delta^{(n)}_t}J^{(n)}(t,c^{(n)}(t))\p rK\sum^{n-1}_{j=0}\int_0^{\Delta^{(n)}_t}{e^{-r(u+j\delta)}p^{(n)}_j(t,c^{(n)}(t),u)du}
\end{align}
with $b^{(n)}(T^{(n-1)}_\delta)=c^{(n)}(T^{(n-1)}_\delta)=K$ and $b^{(n)}(t)\le K\le c^{(n)}(t)$ for $t\in[0,T^{(n-1)}_\delta]$.
\end{theorem}

%%%%%%%%%%%%%%%%%%%%%%%%%%%%%%%%%%%%%%%%%%%%%%%%%%%%%%%%%%%%%%%%%%%%%%%%%%%%%%%
%%% Appendix %%%
%%%%%%%%%%%%%%%%%%%%%%%%%%%%%%%%%%%%%%%%%%%%%%%%%%%%%%%%%%%%%%%%%%%%%%%%%%%%%%%

\appendix
\section{Appendix}
\renewcommand{\theequation}{A-\arabic{equation}}

\begin{proof}[Proof of Proposition \ref{rem:reguC^2}.]
First we show that the process
\begin{align}
Y_s:=e^{-rs}R(t\p s,X_s)\m rK\int_0^se^{-ru}f(t\p u,X_u)du\, ,\qquad u\in[0,T_\delta-t]
\end{align}
is a continuous martingale. Continuity is easily verified by continuity of $V^{(1)}$ and $f$, whereas the martingale property can be checked as follows, by Markov property and \eqref{chvarV}:
\begin{align*}
\EE_x\Big[e^{-rs}R(t\p s,X_s)\Big]=&\EE_x\Big[e^{-r(s+\delta)}\EE_{X_s}\Big[V^{(1)}(t\p s\p \delta,X_{\delta})\Big]\Big]\\
=&\EE_x\Big[e^{-r(s+\delta)}V^{(1)}(t\p s\p \delta,X_{s+\delta})\Big]\\
=&\EE_x\Big[e^{-r\delta}V^{(1)}(t\p \delta,X_{\delta})\m rK\int_\delta^{\delta+s}e^{-ru}I (X_u<b^{(1)}(t\p u)) du\Big].
\end{align*}
Now changing variables in the integral, taking iterated expectations and using Markov property we finally get
\begin{align*}
\EE_x\Big[e^{-rs}R(t\p s,X_s)\Big]=&R(t,x)\m rK\int_0^se^{-r(u+\delta)}\EE_x\big[I (X_{\delta+u}<b^{(1)}(t\p \delta\p u))\big]du\\
=&R(t,x)\m rK\int_0^se^{-ru}\EE_x\big[e^{-r\delta}\PP_{X_u}(X_\delta<b^{(1)}(t\p\delta\p u))\big]du\\
=&R(t,x)\m rK\EE_x\Big[\int_0^se^{-ru}f(t\p u,X_u)du\Big].
\end{align*}
Hence $Y$ is a martingale as claimed.

To prove \eqref{LR00} let $\cD\subset (0,T_\delta)\times(0,\infty)$ be an arbitrary rectangular, open, bounded domain with parabolic boundary $\partial_P\cD$. Since $R\in C([0,T_\delta]\times(0,\infty))$ it is well known (cf.~for instance \cite[Thm.~9, Sec.~4, Ch.~3]{Fri}) that the problem
\begin{align}\label{Cauchy}%\hs{+3pc}
u_t+\L_Xu-ru=-rKf\quad\text{on $\cD$ with $u=R$ on $\partial_P\cD$}
\end{align}
admits a unique classical solution $u_{\cD}\in C^{1,2}(\cD)\cap C(\overline{\cD})$. For $(t,x)\in\cD$ and $\tau_\cD$ the first exit time of $(t\p s,X_s)$ from $\cD$ we can apply Dynkin formula to obtain
\begin{align*}
u_{\cD}(t,x)=&\EE_x\Big[e^{-r\tau_\cD}u_\cD(t\p \tau_\cD,X_{\tau_\cD})+rK\int_0^{\tau_\cD}e^{-rs}f(t\p s,X_s)ds\Big]\\
=&\EE_x\Big[e^{-r\tau_\cD}R(t\p \tau_\cD,X_{\tau_\cD})+rK\int_0^{\tau_\cD}e^{-rs}f(t\p s,X_s)ds\Big]=R(t,x)
\end{align*}
where the last equality follows by the martingale property proved above.
Therefore $u_{\cD}=R$ on $\overline{\cD}$ and by arbitrariness of $\cD$ one has $R\in C^{1,2}((0,T_\delta)\times(0,\infty))$. Finally \eqref{LR00} and \eqref{3.2} imply \eqref{ltsf-3}.
\end{proof}
\vs{4pt}

\begin{proof}[Proof of eq.~\eqref{3.31}]

Since $C^{(2)}$ is a non empty open set we can consider an open, bounded rectangular domain $\mathcal{D}\subset C^{(2)}$ with parabolic boundary $\partial_P\mathcal{D}$. Then the following boundary value problem
\begin{align}\label{Cauchy2}%\hs{+3pc}
u_t+\L_Xu-ru=0\quad\text{on $\cD$ with $u=V^{(2)}$ on $\partial_P\cD$}
\end{align}
admits a unique classical solution $u\in C^{1,2}(\cD)\cap C(\overline{\cD})$ (cf.~for instance \cite[Thm.~9, Sec.~4, Ch.~3]{Fri}). Fix $(t,x)\in\cD$ and denote $\tau_\cD$ the first exit time of $(t+s,X^x_s)_{s\ge 0}$ from $\cD$. Then Dynkin's formula gives
\begin{align*}
u(t,x)=\EE e^{-r\tau_\cD}u(t\p \tau_\cD,X^x_{\tau_\cD})=\EE e^{-r\tau_\cD}V^{(2)}(t\p \tau_\cD,X^x_{\tau_\cD})=V^{(2)}(t,x)
\end{align*}
where the last equality follows from the fact that $e^{-r s\wedge\tau^*}V^{(2)}(t\p (s\wedge\tau^*),X^x_{s\wedge\tau^*})$, $s\ge 0$ is a martingale according to standard optimal stopping theory and $\tau_\cD\le \tau^*$, $\PP$-a.s.
\end{proof}
%%%%%%%%%%%%%%%%%%%%%%%%%%%%%%%%%%%%%%%%%%%%%%%%%%%%%%%%%%%%%%%%%%%%%%%%%%%%%%%
\noindent\ackn{The first named author was supported by EPSRC grant EP/K00557X/1. Both authors are grateful to G.~Peskir for many useful discussions}

%%%%%%%%%%%%%%%%%%%%%%%%%%%%%%%%%%%%%%%%%%%%%%%%%%%%%%%%%%%%%%%%%%%%%%%%%%%%%%%
%%% References %%%
%%%%%%%%%%%%%%%%%%%%%%%%%%%%%%%%%%%%%%%%%%%%%%%%%%%%%%%%%%%%%%%%%%%%%%%%%%%%%%%

\begin{center}

\end{center}
\newpage

%%%%%%%%%%%%%%%%%%%%%%%%%%%%%%%%%%%%%%%%%%%%%%%%%%%%%%%%%%%%%%%%%%%%%%%%%%%%%%%
%%% Affiliations %%%
%%%%%%%%%%%%%%%%%%%%%%%%%%%%%%%%%%%%%%%%%%%%%%%%%%%%%%%%%%%%%%%%%%%%%%%%%%%%%%%

%\par \leftskip=24pt

%\ni Tiziano De Angelis \\
%School of Mathematics \\
%The University of Manchester \\
%Oxford Road \\
%Manchester M13 9PL \\
%United Kingdom \\
%\texttt{tiziano.deangelis@manchester.ac.uk}
%\vs{+2pc}

%\par \rightskip=24pt

%\ni Yerkin Kitapbayev \\
%School of Mathematics \\
%The University of Manchester \\
%Oxford Road \\
%Manchester M13 9PL \\
%United Kingdom \\
%\texttt{yerkin.kitapbayev@manchester.ac.uk}

\end{document}